\documentclass[12pt,reqno]{amsart}
\newtheorem{theorem}{Theorem}[section]
\newtheorem{lemma}[theorem]{Lemma}
\newtheorem{corollary}[theorem]{Corollary}
\newtheorem{proposition}[theorem]{Proposition}

\theoremstyle{definition}
\newtheorem{definition}[theorem]{Definition}
\newtheorem{example}[theorem]{Example}

\theoremstyle{remark}
\newtheorem{remark}[theorem]{Remark}

\numberwithin{equation}{section}

\usepackage{bm,cite}
\usepackage{graphicx}
\usepackage{enumerate}
\usepackage{subfig}
\usepackage{algorithm}  
\usepackage{algorithmic}  
\usepackage{soul}
\usepackage{mathrsfs}
\begin{document}
\title{Optimal transportation for the far-field reflector problem}



\author{Gang Bao}
\address{School of Mathematical Sciences, Zhejiang University, Hangzhou 310027, China.}
\curraddr{}
\email{baog@zju.edu.cn}
\thanks{}

\author{Yixuan Zhang}
\address{School of Mathematical Science, Peking University, Beijing 100871, China.}
\curraddr{}
\email{yixuan@pku.edu.cn}

\thanks{}
\subjclass[2010]{Primary }


\date{}

\dedicatory{}

\begin{abstract}
The inverse reflector problem aims to design a freeform reflecting surface that can direct the light from a specified source to produce the desired illumination in the target area, which is significant in the field of geometrical non-imaging optics. Mathematically, it can be formulated as an optimization problem, which is exactly the optimal transportation problem (OT) when the target is in the far field. The gradient of OT is governed by the generalized Monge-Ampère equation that models the far-field reflector system. Based on the gradient, this work presents a Sobolev gradient descent method implemented within a finite element framework to solve the corresponding OT. Convergence of the method is established and numerical examples are provided to demonstrate the effectiveness of the method.

\end{abstract}

\maketitle
\section{Introduction}

\subsection{Far-field reflector problem}

The far-field reflector system consists of a point source of the light placed at the origin $O$, a perfectly reflecting surface $\Gamma$ that is radial relative to $O$, and a far-field target area that receives the light originating from the source. The light source has the power density $f(x)$ in the direction $x\in \Omega\subset \mathbb{S}^{2}$, which describes the distribution of the intensity of rays from the source $O$. The reflecting surface $\Gamma$ can be characterized by  
\[\Gamma :=\left\{x\rho(x):x\in \Omega,\,\rho>0\right\}.\]
For a ray originating from $O$ in the direction $x$, it hits the surface $\Gamma$ at the point $x\rho(x)$ and produces a reflected ray in the direction $y\in \mathbb{S}^{2}$, which follows Snell's law
\begin{equation}\label{reflection_map}
    y = T(x):= x-2(x\cdot\mathrm{n}) \mathrm{n},
\end{equation}
where $\mathrm{n}$ is the outward normal of $\Gamma$ at the point $x\rho(x)$. Over all directions $x\in \Omega$, this reflecting procedure creates a far-field illumination intensity $g(y)$ on the domain $\Omega^{*}\subset \mathbb{S}^{2}$. Suppose there is no loss of energy in the reflection. Then by the energy conservation law, the map $T$ defined by (\ref{reflection_map}) is measure preserving, i.e., 
\begin{equation}\label{mp}
    \int_{T^{-1}(B)}f(x)\mathrm{d}x=\int_{B}g(y)\mathrm{d}y, \quad \forall\text{ Borel set }B\subset \Omega^{*},
\end{equation}
which is denoted by $T_{\#} f = g$, the pushforward measure of $f$. Suppose that the source intensity $f$ and the target intensity $g$ have been given in advance. For the above far-field reflector system, our goal is to reconstruct the reflecting surface $\Gamma$ that is capable of producing the prescribed illumination distribution $g(y)$ on the region $\Omega^{*}$ of a far-field sphere (see Figure \ref{system}).

\begin{figure*}[h]
    \centering
    \includegraphics[width=0.6\textwidth]{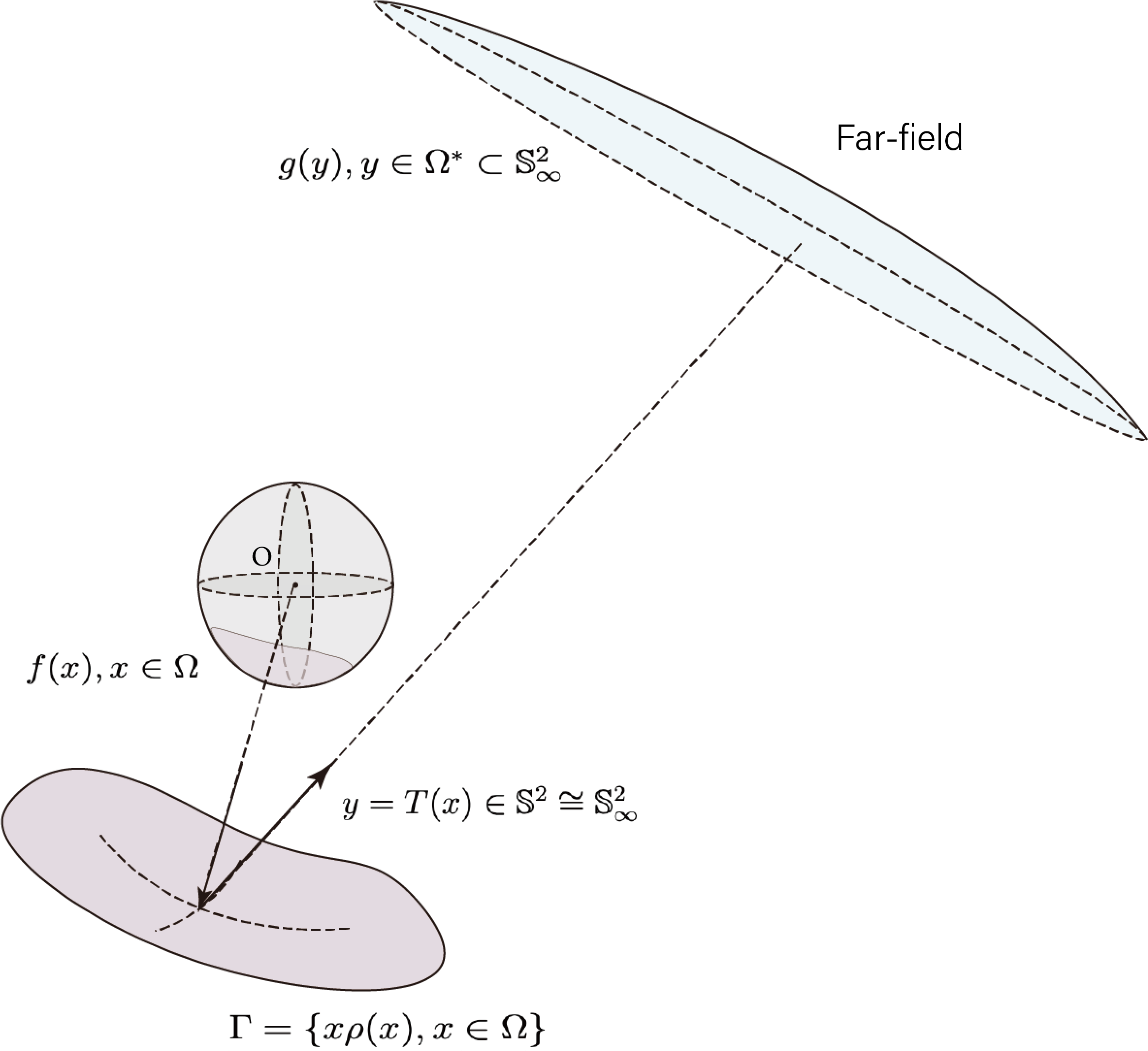}
    \caption{Far-field reflector system. The reflecting surface $\Gamma$ receives the source intensity $f(x)$ and produces the target intensity $g(y)$, where $x,y \in \mathbb{S}^{2}$. }\label{system} 
\end{figure*}

The measure-preserving property (\ref{mp}) can be expressed as the Jacobian equation following by applying the change of variables,
\begin{equation}\label{c_JO}
g\left(T(x)\right) \left|\operatorname{det}(\nabla T)\right| = f(x).
\end{equation}
The map $T$ depends on the radial distance function $\rho$ through the relation $\mathrm{n} =(\nabla\rho-\rho x)/\sqrt{\rho^{2}+|\nabla \rho|^2}$. Hence by (\ref{reflection_map}),
\begin{equation}\label{reflection_map2}
    T(x) = \frac{\nabla \rho - \left(\eta -\rho\right)x}{\eta}, \qquad \eta=\left(|\nabla \rho|^2+\rho^2\right) / 2 \rho.
\end{equation}
Through (\ref{c_JO}), \cite{wang1996design} yields a fully nonlinear partial differential equation in local coordinates,
\begin{equation}\label{far}
\begin{aligned}
    \frac{1}{\eta^{2}} \operatorname{det}\left(-D^2 \rho+\frac{2}{\rho} D \rho \otimes D \rho-(\rho -\eta)I\right)&=\frac{f(x)}{g(T(x))}, \qquad x\in\Omega,\\
    T(\Omega) &= \Omega^*.
\end{aligned}
\end{equation}
The existence, uniqueness and regularity of (\ref{far}) have been studied in \cite{caffarelli2008regularity, wang1996design, wang2004design}.

The numerical strategies for solving the far-field reflector problem can be roughly categorized as the ray-mapping method \cite{bosel2017single, desnijder2019ray, feng2016freeform, Fournier2010fast}, the supporting quadric method \cite{de2016far, doskolovich2022supporting, kochengin1998determination, kochengin2003computational, meyron2018light}, and the optimal transport method. The ray-mapping method consists in building a ray mapping between the source and the target, and then constructing the surface that is capable of producing this map following (\ref{reflection_map2}). However, an integrability condition should be satisfied to ensure the existence of the surface, which is difficult to verify theoretically. For the supporting quadric method, the optical surface is constructed by taking the envelope of paraboloids of revolution that are controlled by points in the target area, leading to a computational complexity up to $O\left(\frac{N^4}{\epsilon} \log \frac{N^2}{\epsilon}\right)$ \cite{kochengin2003computational}.

It was proved in \cite{glimm2003optical, wang2004design} that the far-field reflector problem is equivalent to the optimal transport problem with a logarithmic type cost function. In fact, this OT is a linear assignment problem, which can be solved using the auction algorithm \cite{bykov2018linear} or the Hungarian algorithm with a complexity of $O(N^3)$. On the other hand, the entropic regularized OT can be solved by the Sinkhorn algorithm \cite{benamou2020entropic}, which reduces the computational complexity to $O(N^2)$. However, the existence of the regularizer severely degrades the accuracy of the solution. Using OT, the equation (\ref{far}) can be simplified into a generalized Monge-Ampère equation. In \cite{romijn2021iterative}, the equation is solved by the least-squares method, which is stable but computationally complex and slow to converge. The work \cite{hamfeldt2021convergent} offered a convergent finite difference method. However, due to the complexity of constructing a monotonic scheme in local coordinates and the slow convergence rate of the first-order Euler method, this approach results in a high computational cost. In \cite{brix2015solving} and \cite{wu2013freeform}, the equation is discretized using the B-spline scheme and the finite difference scheme, respectively, and solved with Newton-type methods. The convergence of these methods is highly dependent on the choice of the initial value and the regularity of the illumination distributions.

\subsection{Main results}

The existing methods focus on solving the standard Monge-Ampère equation $\operatorname{det}\left(I+D^2u\right)= f$ on the plane \cite{benamou2014numerical, feng2009mixed, loeper2005numerical, prins2015least, sulman2011efficient}. In \cite{awanou2015standard}, a first-order relaxation method has been proposed to solve the standard Monge-Ampère equation, given by
\begin{equation}\label{c_des}
    u^{k+1} = u^{k} + \Delta^{-1}\text{MA}\left(x, \nabla u^{k},D^{2}u^{k}\right),
\end{equation}
where $\text{MA}(x,\nabla u, D^{2}u)=f-\operatorname{det}\left(I+D^2u\right)$. In particular, (\ref{c_des}) is implemented using the finite element method, and its convergence is established by demonstrating that \eqref{c_des} constitutes a fixed point scheme. Later, \cite{mcrae2018optimal} extended this approach numerically to solve the $L^2$-Monge-Ampère equation on the sphere using the finite element method, addressing the challenge of mesh adaptivity. A similar approach was employed in \cite{bao2024optimal} to compute the quadratic Wasserstein distance for applications in inverse problems. However, the generalized Monge-Ampère equation associated with (\ref{far}) is significantly more complicated, where 
\begin{equation}\label{GMA1}
\text{MA}(x,\nabla u, D^{2}u)=f(x)-\tilde{g}(x,\nabla u)\operatorname{det}\left(D^2u+A(x,\nabla u)\right).
\end{equation}
The high nonlinearity of the equation makes the previous fixed point method ineffective for proving the convergence of (\ref{c_des}) in the case of (\ref{GMA1}). Moreover, no convergence analysis is available, further complicating its numerical treatment.

In this work, we transform the far field reflector problem (\ref{far}) into an optimal transport problem, which is equivalent to a generalized Monge-Ampère equation. We apply (\ref{c_des}) to solve this equation and rigorously establish the convergence analysis of the method. By introducing the duality theory of OT, we demonstrate that the generalized Monge-Ampère operator in (\ref{GMA1}) corresponds to the gradient of the dual functional of OT. Consequently, $(-\Delta)^{-1}\text{MA}$ is identified as the $\dot{H}^{1}$ gradient in this context, and (\ref{c_des}) is shown to be a descent scheme that guarantees the strict decrease of the dual functional, ultimately converging to the solution of the equation. Based on this, we refer to this approach as the Sobolev gradient method. It should be noted that our analysis also provides a novel convergence proof for the scheme (\ref{c_des}) in the context of Monge-Ampère equations discussed in \cite{awanou2015standard, mcrae2018optimal}. Numerical experiments of this approach for the far-field problem can be implemented by adapting the finite element method from \cite{mcrae2018optimal}.

There are several advantages associated with the proposed method. First, (\ref{c_des}) is a Poisson equation, which can be solved by numerous fast algorithms with a computational complexity as low as $O(N)$. Both the theoretical proof and the numerical results can verify that the Sobolev gradient speeds up the convergence. Besides, different from the Newton-type methods, (\ref{c_des}) does not need to compute the Jacobian matrix, which could be nontrivial for singular intensities. Compared with other methods, the approach proposed in this work is easier to accomplish, avoiding the complicated pre-calculations of numerical schemes. It provides an efficient way to solve the freeform reflector problem.

The paper is organized as follows. Section 2 introduces the relation between the optimal transport problem and the far-field reflector design. Section 3 derives the form of (\ref{GMA1}), which is related to the gradient of the dual functional of OT. Section 4 is devoted to proving the convergence of the Sobolev gradient descent scheme (\ref{c_des}). In Section 5, the numerical algorithm for solving the far-field reflector problem is introduced. Section 6 provides some numerical examples to demonstrate the feasibility and effectiveness of the method.

\subsection*{Notation and preliminaries}
For a matrix $\omega$, its subscripts $\omega_{ij}$ denotes the $(i,j)^{th}$ entry of $\omega$ and superscripts $\omega^{ij}$ denotes the $(i,j)^{th}$ entry of $\omega^{-1}$. For a matrix function $\omega(x)\in C(\Omega)^{n\times n}$, we define its norm $\|\omega\|_{p,\Omega}$ by
\begin{equation}\label{norm_defn}
   \|\omega\|_{p,\Omega}: =\sup_{x\in\Omega} \|\omega(x)\|_{p}, \quad \text{where }\|\omega(x)\|_{p}:= \sup_{b\in\mathbb{R}^{n}/0} \frac{\|\omega(x)\cdot b\|_{p}}{\|b\|_{p}}
\end{equation}
is induced by the $p$-norm for vectors.

The notation $D_{i}$ denotes the $i^{th}$ derivative in local coordinates, and the variable is specified in the subscript when the derivative operates on a multivariable function. In particular, for the cost function $c(x,y)$ in the optimal transport problem, we use the notation
\begin{equation}\label{deri_defn}
    c_{i j \cdots, k l \cdots}:=D_{x_i} D_{x_j} \cdots D_{y_k}D_{y_l} \cdots c,
\end{equation}
to simplify the notation of derivatives. In addition, the superscripts $c^{i,j}$ denotes the $(i,j)^{th}$ entry of the inverse matrix of $D_{xy}^{2}c= \left(c_{i,j}^{}\right)$. The function space
\begin{equation}
    \tilde{C}^{k,\alpha}(\Omega):=\left\{u\in C^{k,\alpha}(\Omega):\int_{\Omega} u(x)\mathrm{d}x= 0 \right\},
\end{equation}
is used to denote the $C^{k,\alpha}$ space with zero mass on the domain.

\section{Freeform design and Optimal Transport}
This section briefly reviews the basic concepts of the optimal transport problem and its connection with the far-field reflector design.

\subsection{Optimal transport problem}
Suppose that $\mathcal{X}$, $\mathcal{Y}$ are complete and separable metric spaces and $f$, $g$ are probability density functions defined on $\mathcal{X}$ and $\mathcal{Y}$, respectively. The cost function $c(x,y):\mathcal{X}\times \mathcal{Y}\rightarrow \mathbb{R}$ indicates the cost of transporting a unit of mass from $x$ to $y$. Then Monge's optimal transport problem seeks the most efficient way of rearranging $f$ into $g$, which is to minimize 
\begin{equation}\label{primal}
     \inf_{T\in \Pi(f, g)} \int_{\mathcal{X} \times \mathcal{Y}} c(x, T(x)) f(x)\mathrm{d} x,\\
\end{equation}
\[\begin{aligned}
\text{where}\qquad \Pi(f,g):=\bigg\{T:\mathcal{X}\rightarrow \mathcal{Y}: &\;\;T_{\#}f =g,\;\text{i.e.,}\\
&\int_{T^{-1}(B)}f(x)=\int_{B}g(y),\quad \forall B\subset \mathcal{Y}\bigg\},
\end{aligned}\]
is the set of all measure-preserving maps. The existence and uniqueness of the optimal map in (\ref{primal}) have been established for different cost functions on the Euclidean space and compact manifolds, including $c(x,y)=\frac{1}{2}|x-y|^2$ on $\mathbb{R}^{d}$ \cite{brenier1991ploar, mccann2001polar}, $c(x,y)= -\log(1-x\cdot y)$ on $\mathbb{S}^{2}$, etc.

The dual formulation of the optimal transport problem is 
\begin{equation}\label{dual}
\begin{aligned}
    &\inf_{(u, v)}\, I(u,v):=\int_\mathcal{X} u(x) f(x)\mathrm{d} x+\int_\mathcal{Y} v(y) g(y)\mathrm{d} y,\\
    &(u,v)\in C(\mathcal{X})\times C(\mathcal{Y}),\quad u(x) +v(y)+ c(x,y)\geq 0.
\end{aligned}
\end{equation}
This is a linear programming problem, which is desirable for designing numerical algorithms.
The standard duality result shows that the infimum in (\ref{primal}) is equal to the negative of the infimum in (\ref{dual}).

\subsection{OT interpretation of the far-field reflector problem}

Considering $\mathcal{X} = \mathcal{Y} = \mathbb{S}^{2}$, the probability density functions $f$ and $g$ are the source intensity and target intensity in the far-field reflector problem, respectively. Then, the far-field problem (\ref{far}) is equivalent to the optimal transport problem.
\begin{theorem}[\cite{glimm2003optical, wang2004design}]\label{OT_linkA}
Suppose that $f$ and $g$ are bounded positive functions on the connected domains $\Omega\subset\mathbb{S}_{-}^{2}$ and $\Omega^*\subset \mathbb{S}_{+}^{2}$ respectively. Then there is a minimizer $(u,v)$ of the dual problem (\ref{dual}) for the cost function $c(x,y)=-\log(1-x\cdot y)$. In particular, $(u,v)$ is unique up to a constant, and $\rho = e^{-u}$ solves the reflector problem (\ref{far}).
\end{theorem}
Therefore, by substituting $\rho = e^{-u}$ into (\ref{reflection_map2}), we can obtain the optical map in terms of the variable $u$,
\begin{equation}\label{reflector_map3}
    T(x) = \frac{(|\nabla u|^{2}-1)x-2\nabla u}{|\nabla u|^{2}+1},
\end{equation}
which depends only on the gradient of $u$. 

In addition to the cost $-\log(1-x\cdot y)$, the far-field problem can also be associated with (\ref{dual}) under the cost $c(x,y) = \log(1-x\cdot y)$.
\begin{theorem}[\cite{wang2004design}]\label{OT_linkB}
Under the same assumptions of Theorem \ref{OT_linkA}, there is a minimizer $(\varphi,\psi)$ of the dual problem (\ref{dual}) for $c(x,y)=\log(1-x\cdot y)$. The minimizer is unique up to a constant and $\rho = e^{\varphi}$ solves the reflector problem (\ref{far}).
\end{theorem}

According to \cite{wang2004design}, the solution of (\ref{far}) must be a constant multiplication of $\rho = e^{-u}$ of Theorem \ref{OT_linkA} or $\rho=e^{\varphi}$ of Theorem \ref{OT_linkB}. In this work, we mainly focus on the case $\rho = e^{-u}$. Here, the minimizer $u$ is Lipschitz continuous but not necessarily $C^1$ smooth, and thus $\rho = e^{-u}$ is understood as a generalized solution \cite{wang1996design} of (\ref{far}). The regularity theorem \cite{ma2005regularity} of OT implies that the smoothness of $u$ depends on the smoothness of $f,g$ and geometric structures of $\Omega$, $\Omega^{*}$. 

The relationship between the far-field reflector problem and the primal form (\ref{primal}) of OT can also be developed, as detailed in Corollary \ref{OT_linkC}. It is a direct result of OT's duality theory. 
\begin{corollary}[\cite{glimm2003optical, wang2004design}]\label{OT_linkC}
Suppose the same conditions as in Theorem \ref{OT_linkA} hold. Then the optical map $T$ given in (\ref{reflector_map3}) is the unique measure-preserving map minimizing the functional (\ref{primal}) for the cost function $c(x,y) = -\log(1-x\cdot y)$.
\end{corollary}

In this work, we aim to solve the freeform design problem for the far-field problem using a gradient descent method minimizing (\ref{dual}).

\section{Gradient method of the dual problem}

According to Theorem \ref{OT_linkA}, by substituting $\rho = e^{-u}$, the problem (\ref{far}) is transformed into a generalized Monge-Ampère equation (\ref{MAO}) in (\ref{GMA1}),
\begin{equation}\label{MAO}
    f(x)-\tilde{g}\left(x,\nabla u\right)\operatorname{det}\left(D^{2}u + A(x, \nabla u)\right) = 0.
\end{equation}
The explicit expressions of $\tilde{g}$ and $A$ are given Proposition \ref{dual_grad} of this section.

Next, we prove that the expression of (\ref{MAO}) is exactly the gradient of the dual functional and thus (\ref{MAO}) can be viewed as the first order optimality condition of (\ref{dual}). Based on this fact, we can design gradient methods for optimizing (\ref{dual}).

Suppose that $\mathcal{X} = \mathcal{Y}$ is the sphere $\mathbb{S}^{n}$ or the flat torus $\mathbb{T}^{n} = \mathbb{R}^n / \mathbb{Z}^n$, $n\in \mathbb{Z}^{+}$. To ensure the regularity of the solution to the optimal transport problem, we assume that the cost function $c(x,y)\in C^{4\times 4}(\mathcal{X}\times \mathcal{Y})$ is locally semi-concave, and satisfies Ma-Trudinger-Wang (MTW) condition in \cite{ma2005regularity, villani2009optimal} as well as the twist condition,
\begin{itemize}
    \item $x\mapsto \nabla_{x}c(x, y)$ is injective for $(x,y)\in\mathcal{X}\times \mathcal{Y}$;
    \item $\operatorname{det}D_{xy}^{2}c(x,y)\neq 0$ on $\mathcal{X}\times\mathcal{Y}$.
\end{itemize}

First, we introduce the notation of $c$-transform, which provides an explicit relationship between $u$ and $v$ in (\ref{dual}), thereby defining the transport map and simplifying the optimization form of (\ref{dual}).
\begin{definition}
For $u:\mathcal{X}\rightarrow \mathbb{R}$ and $v:\mathcal{Y}\rightarrow \mathbb{R}$, their $c$-transforms are defined by
    \begin{equation}\label{c_trans}
     \begin{aligned}
     u^{c}(y) = \sup_{x\in \mathcal{X}} -c(x,y) - u(x)\\
     v^{c}(x) = \sup_{y\in \mathcal{Y}} -c(x,y) - v(y)
     \end{aligned}
    \end{equation}
A function $u$ (or $v$) is said to be $c$-convex if $u= u^{cc}$ (or $v= v^{cc}$). 
\end{definition}
It can be verified that $c$-convex functions are Lipschitz continuous \cite{villani2009optimal}. For a $c$-convex function $u$ and any $x_{0}\in \mathcal{X}$, there exists $y_{0}\in \mathcal{Y}$ such that  
\begin{equation}\label{opps}
\begin{aligned}
\begin{cases}
u(x_{0}) + u^{c}(y_{0}) + c(x_{0}, y_{0}) &=0,\\
u(x) + u^{c}(y_{0}) + c(x,y_{0}) &\geq 0,\quad \forall x\in \mathcal{X}.
\end{cases}
\end{aligned}
\end{equation}
If $u\in C^2(\mathcal{X})$, the optimality conditions of (\ref{opps}) imply that  
\begin{equation}\label{optimal_cond1}
    \nabla u (x_{0}) + \nabla_{x} c(x_{0},y_{0}) =0,
\end{equation}
\begin{equation}\label{optimal_cond2}
    D^{2}u(x_{0}) + D_{xx}^{2}c(x_{0},y_{0}) \geq 0.
\end{equation}
From (\ref{optimal_cond1}), we can define a map $T_{u}:\mathcal{X}\rightarrow \mathcal{Y}$ by
\begin{equation}\label{optimal_map}
    T_{u}(x) = \mathcal{T}(x,\nabla u), \quad \text{where }\mathcal{T}(x,p):= (\nabla_{x}c)^{-1}(x,-p),
\end{equation}
which maps each $x_0$ to the corresponding $y_0$. In fact, based on (\ref{optimal_cond2}), we can also propose a sufficient condition to ensure the $c$-convexity of $u\in C^{2}(\mathcal{X})$, stated in the following lemma.
\begin{lemma}[\cite{kim2012parabolic}]
Suppose that $u\in C^{2}(\mathcal{X})$ and
\begin{equation}\label{c_convex}
    D^{2}u(x) + A\left(x,\nabla u(x)\right) > \lambda ,\quad \forall x\in \mathcal{X},
\end{equation}
for some $\lambda>0$, where $A(x,p):=D_{xx}^{2}c(x,\mathcal{T}(x,p))$. Then $u$ is $c$-convex and $T_{u}$ in (\ref{optimal_map}) is a diffeomorphism. The equality $ u(x)+u^{c}(y)+c(x,y)=0$ holds if and only if $y= T_{u}(x)$.
\end{lemma}
In general, $u\in C^2(\mathcal{X})$ is said to be strictly $c$-convex with respect to $\lambda$ if (\ref{c_convex}) is satisfied. For any function $u$, since $I(u^{cc},u^c)\leq I(u,u^{c}) \leq I(u,v)$ , it is possible to reduce (\ref{dual}) to the optimization problem depending on a single variable $u$, see \cite{villani2009optimal}.

\begin{proposition}\label{dual_grad}
The dual problem (\ref{dual}) is equivalent to maximizing the functional $\mathcal{J}$ over the set of $c$-convex functions:
\begin{equation}\label{reduced_dual}
        \mathcal{J}(u): = \int_{\mathcal{X}}u(x)f(x)\mathrm{d} x + \int_{\mathcal{Y}} u^{c}(y)g(y)\mathrm{d} y.
     \end{equation}
     The functional $\mathcal{J}$ is convex and Lipschitz continuous. Moreover, if $u\in C^{2}(\mathcal{X})$ is strictly $c$-convex, the first variation of $\mathcal{J}$ at $u$ is given by
     \begin{equation}\label{gra_org}
         \mathcal{J}^{\prime}(u) = f - g(T_{u})\left|\operatorname{det}(\nabla T_{u})\right|, 
     \end{equation}
     where $T_{u}=\mathcal{T}(x, \nabla u)$ is the map defined in (\ref{optimal_map}). The formula (\ref{gra_org}) is equal to
     \begin{equation}\label{gra_now}
         \mathcal{J}^{\prime}(u) = f - \tilde{g}\left(x,\nabla u\right)\operatorname{det}\left(D^{2}u + A(x, \nabla u)\right),
     \end{equation}
     where $A(x,p)$ is given in (\ref{optimal_map}) and 
     \[\tilde{g}(x,p):= \frac{g(\mathcal{T}(x,p))}{\left|\operatorname{det}\left(D_{xy}^{2}c\left(x,\mathcal{T}(x,p)\right)\right)\right|}.\] 
\end{proposition}
If the $c$-convex function $u\in C^{2}(\mathcal{X})$ achieves the infimum of the functional $\mathcal{J}$, then $T_{u}$ is exactly the optimal transport map solving (\ref{primal}) and the generalized Monge-Ampère equation (\ref{MAO}) holds. For $c(x,y) = -\log(1-x\cdot y)$, it can be checked that the optimal map $T_{u}$ computed through (\ref{optimal_map}) coincides the optical map (\ref{reflector_map3}).

\begin{proof}[Proof of Proposition \ref{dual_grad}]
The convexity of the functional follows from
\[\begin{aligned}\left(t u_1+(1-t) u_2\right)^c(y) & =\sup _{x \in \mathcal{X}}\left(-c(x,y)-t u_1(x)-(1-t) u_2(x)\right) \\ 
& \leq \sup _{x \in \mathcal{X}} t\left(-c(x,y)-u_1(x)\right)+\sup _{x \in \mathcal{X}}(1-t)\left(-c(x,y)-u_2(x)\right) \\ 
& =t u_1^c(y)+(1-t) u_2^c(y).
\end{aligned}\]

For any $u_1$, $u_2\in C(\mathcal{X})$, 
\[\begin{aligned} 
u_1^c(y)=\sup _x-c(x, y)-u_1(x) & \geq \sup _x-c(x, y)-u_2(x)-\left\|u_2-u_1\right\|_{C^0} \\ 
& =u_2^c(y)-\left\|u_2-u_1\right\|_{C^0},\end{aligned}\]
which means that 
$$\|u_1^c-u_2^c\|_{C^0}\leq \|u_2-u_1\|_{C^0}.$$
Further, we obtain the Lipschitz continuity of $\mathcal{J}$ on $C(\mathcal{X})$,
\[\left|\mathcal{J}\left(u_1\right)-\mathcal{J}\left(u_2\right)\right| \leq\|f\|_{L^1}\left\|u_1-u_2\right\|_{C^0}+\|g\|_{L^1}\left\|u_1^c-u_2^c\right\|_{C^0} \leq 2\left\|u_1-u_2\right\|_{C^0}.\]

Considering the strictly $c$-convex function $u\in C^{2}(\mathcal{X})$. Let $v\in C(\mathcal{X})$ and $\epsilon$ be a small enough constant. For fixed $y\in \mathcal{Y}$, we introduce the notation $x:=T_{u}^{-1}(y)$ and 
\[x_{\epsilon}\in \operatorname{argmax}_{x\in \mathcal{X}}\left\{-c(x,y)-(u+\epsilon v)(x)\right\}.\]
Using the property of the $c$-transform, we obtain
\[\begin{aligned}
        (u+\epsilon v)^{c}(y) - u^{c}(y) &= -c(x_{\epsilon},y)-(u+\epsilon v)(x_{\epsilon})-u^{c}(y)\\
        &\leq u(x_{\epsilon})- (u+\epsilon v)(x_{\epsilon})\\
        &= -\epsilon v(x_{\epsilon}),\\
        (u+\epsilon v)^{c}(y) - u^{c}(y) &= (u+\epsilon v)^{c}(y) + c(x_0,y) + u(x_0)\\
        &\geq -(u+\epsilon v)(x_0)+u(x_0) \\
        &= -\epsilon v(x_0).
        \end{aligned}\]
It follows that
\[\left|\frac{(u+\epsilon v)^{c}(y) - u^{c}(y)}{\epsilon} + v(x)\right|\leq\left|v(x_{\epsilon}) - v(x)\right|.\]
Every convergent subsequence of $\{x_{\epsilon}\}$ converges to the point $x$, thus $\{x_{\epsilon}\}\rightarrow x$ as $\epsilon \rightarrow 0$. Therefore, $v(x_{\epsilon})\rightarrow v(x)$ and 
\[\lim_{\epsilon\rightarrow 0}\frac{(u+\epsilon v)^{c}(y) - u^{c}(y)}{\epsilon}=-(v\circ T_{u}^{-1})(y).\]

Hence we have
\[
\begin{aligned}
\mathcal{J}^{\prime}(u)v=\lim _{\epsilon \rightarrow 0} \frac{\mathcal{J}(u+\epsilon v)-\mathcal{J}(u)}{\epsilon}&=\int_{\mathcal{X}} v(x) f(x)\mathrm{d}x+\int_{\mathcal{Y}} \lim _{\epsilon \rightarrow 0} \frac{(u+\epsilon v)^c(y)-v^c(y)}{\epsilon} g(y)\mathrm{d}y\\
&= \int_{\mathcal{X}}v(x)f(x)\mathrm{d}x -\int_{\mathcal{Y}}(v\circ T_{u}^{-1})(y)g(y)\mathrm{d}y \\
&= \int_{\mathcal{X}}v(x)\big(f-g(T_{u})|\operatorname{det}(\nabla T_{u})|\big)\mathrm{d}x.
\end{aligned}\]
Thus we obtain $\mathcal{J}^{\prime}(u)$ in (\ref{gra_org}). A a direct computation  \cite{villani2009optimal, loeper2011regularity} of (\ref{gra_org}) yields the formula (\ref{gra_now}). 

\end{proof}

Furthermore, the operator $\mathcal{J}^{\prime}(u)$ is differentiable on the space $C^{2}(\mathcal{X})$, which gives the second order information of $\mathcal{J}$.
\begin{corollary}\label{hessian}
The second order derivative of the functional $\mathcal{J}$ is given by the linearized operator $\mathcal{L}$ of $\mathcal{J}^{\prime}$ at $u$,
\begin{equation}\label{linear_op}
    \mathcal{L}_{u}v:= \sum_{i}\mathcal{L}_{u}^{i}\,D_{i}v + \sum_{ij} \mathcal{L}_{u}^{ij}\, D_{_{ij}}^{2}v, \quad v\in C^{2}(\mathcal{X}),
\end{equation}
\[\mathcal{L}_{u}^{i}:= \operatorname{det}(\omega)\sum_{kl} \left(D_{p_{i}}\tilde{g} + \tilde{g}\omega^{kl}D_{p_{i}}A_{kl}\right), \quad
      \mathcal{L}_{u}^{ij}:= \tilde{g}\operatorname{det}(\omega)\,\omega^{ij},\]
where $\omega:= D^{2}u+A(x,\nabla u)$ and $D_{p_{i}}A_{kl}$, $D_{p_{i}}\tilde{g}$ are evaluated at the point $(x,p) = \left(x,\nabla u\right)$. Here,
\begin{equation}\label{partial_derivatives}
    \begin{aligned}
    &D_{p_{i}}A_{kl}(x,p) = \sum_{j}c_{kl, j}c^{j,i},\\
    &D_{p_{i}}\tilde{g}(x,p) = \left|\operatorname{det}(D_{xy}^{2}c)\right|^{-1}\left(\sum_{j}c^{j,i} \left(D_{j}g\circ \mathcal{T}\right) + \left(g\circ \mathcal{T}\right)\sum_{jkl}c^{j, k}c_{k, jl}c^{l,i}\right).
    \end{aligned}
\end{equation}
where the cost function $c$ and its derivatives are evaluated at $(x,y) = \left(x,\mathcal{T}(x,p)\right)$.
\end{corollary}

For $c(x,y) = \frac{1}{2}|x-y|^2$ on $\mathbb{R}^d$, the equation (\ref{MAO}) is $g(x+\nabla u)\operatorname{det}(D^2u +I)=f(x)$, which can be solved by Newton's method using (\ref{linear_op}). However, for the far-field reflector problem $c(x,y) = -\log(1-x\cdot y)$ on $\mathbb{S}^2$, due to the high nonlinearity of the equation and the domain, $\mathcal{L}$ easily becomes nontrivial. The gradient method is preferred for this case, i.e., 
\[u^{k+1}=u^k - \tau \cdot \mathcal{J}^{\prime}(u^k).\] 
However, the step size $\tau$ is hard to choose to ensure the convergence of the method. Motivated by the Sobolev gradient \cite{Neuberger}, we can define a new gradient $\mathcal{J}_{s}^{\prime}$ by changing the underlying inner space from $L_2$ to $\dot{H}^1$ (the seminorm of $H^1$),
\[\left\langle\mathcal{J}^{\prime}, \eta \right\rangle_{L^2} = \left\langle\mathcal{J}_{s}^{\prime},\eta\right\rangle_{\dot{H}^{1}}:= \left\langle \nabla \mathcal{J}_{s}^{\prime},\nabla \eta\right\rangle_{L^2}.\]
More formally, we define 
\[\mathcal{J}_{s}^{\prime}:=(-\Delta)^{-1}\mathcal{J}^{\prime},\]
which yields the Sobolev gradient descent scheme of (\ref{reduced_dual}),
\begin{equation}\label{sobolev_des}
    \Delta u^{k+1} = \Delta u^{k} +\tau \cdot\mathcal{J}^{\prime}(u^{k}).
\end{equation}

In this work, we utilize the descent scheme (\ref{sobolev_des}) to develop the numerical algorithm for solving the far-field problem.

\section{Convergence of the Sobolev descent scheme}

This section is devoted to the convergence analysis of the Sobolev gradient descent sequence (\ref{sobolev_des}) for the optimization problem (\ref{reduced_dual}). 

We assume that the conditions stated at the beginning of Section 3 are satisfied. We begin with the following regularity theorem \cite{caffarelli1990interior, kim2012parabolic, loeper2009regularity, loeper2011regularity, ma2005regularity, villani2009optimal} of the generalized Monge-Ampère equation (\ref{MAO}).

\begin{theorem}\label{existence}
 Suppose that $f,g \in C^{1,1}(\mathcal{X})$ have positive lower bounds on $\mathcal{X}$. Then the generalized Monge-Ampère equation (\ref{MAO}) has a unique $c$-convex solution $u\in \tilde{C}^{3,\alpha}(\mathcal{X})$. In addition, $u$ is the unique $c$-convex minimizer of $\mathcal{J}(u)$ in (\ref{reduced_dual}).
\end{theorem}

For this solution $u$, (\ref{optimal_cond2}) implies that $\omega = D^{2}u +A(x,\nabla u)$ is a positive semi-definite matrix. Considering the boundedness of $\|u\|_{C^2}$, there exists a constant $C>0$ such that 
\begin{equation}\label{b1}
    \|\omega\|_{2,\mathcal{X}}<C.
\end{equation}
Since $c\in C^4$ and $\operatorname{det}(D_{xy}^2 c)\neq 0$, we know that $\left|\operatorname{det} D_{x y}^2 c\right|$ has a positive lower bound on $\mathcal{X}\times \mathcal{Y}$. Therefore, the boundedness of $f,g$ in Theorem \ref{existence} indicates that
\begin{equation}\label{b2}
\operatorname{det}(\omega)=\frac{f(x)}{\tilde{g}(x, \nabla u(x))}>C, \quad \forall x \in \mathcal{X},
\end{equation}
for some $C>0$. The inequality (\ref{b1}) and (\ref{b2}) together imply that the eigenvalues of $\omega$ are bounded from above and below, i.e., there exists some $\lambda >0$ such that 
\begin{equation}\label{lambda_defn}
    \frac{2}{\lambda} < D^{2}u(x)+ A\left(x,\nabla u(x)\right)< \frac{1}{2}\lambda,\quad \forall x\in\mathcal{X}.
\end{equation}
Obviously, there exists a neighborhood $V_{\lambda}$ of $u$ with the radius $r$ 
\begin{equation}\label{V_lambda}
    V_{\lambda}:=\left\{v\in \tilde{C}^{2}(\mathcal{X}):\|u-v\|_{C^2(\mathcal{X})}\leq r(\lambda)\right\},
\end{equation}
such that a function $v$ in $V_{\lambda}$ is strictly $c$-convex with respect to $\lambda^{-1}$,
\[ \frac{1}{\lambda} < D^{2}v(x)+ A\left(x,\nabla v(x)\right)< \lambda, \quad \forall x\in \mathcal{X}.\]

If the descent sequence is contained in this neighborhood of the solution $u$, then we can prove the strictly decreasing of $\mathcal{J}$ and the local convergence of $\{u^k\}_k$. To prove that, we need the following result to bound the divergence of the cofactor matrix of $\omega$ by the parameter $\lambda$ in (\ref{lambda_defn}).

\begin{lemma}\label{prior_estimate}
Suppose that $u$ is an arbitrary strictly $c$-convex function. For the map $T_{u}$ defined in (\ref{optimal_map}), its Jacobi matrix has divergence-free rows, i.e.
\begin{equation}\label{divergence_free}
    \sum_{j=1}^{n}D_{j}\left((\operatorname{cof}\nabla T_{u})_{ij}^{ }\right)=0, \qquad i=1,\cdots,n.
\end{equation}
where $\operatorname{cof}\nabla T_{u}$ denotes the cofactor matrix of $\nabla T_{u}$. In addition, for any $x\in \mathcal{X}$, the matrix $\omega = D^{2}u(x) + A\left(x,\nabla u(x)\right)$ satisfies 
\begin{equation}\label{divergence_estimate}
    \quad \left|\sum_{j=1}^{n}D_{j}\left((\operatorname{cof}\omega)_{ij}^{}\right)\right|\leq C\left\|\operatorname{cof}\omega\right\|_{2,\mathcal{X}}\cdot\left(\left\|\omega\right\|_{2,\mathcal{X}}+1 \right),\qquad i =1,\cdots, n,
\end{equation}
where the constant $C$ depends only on the cost function $c$ and the dimension $n$.
\end{lemma}

\begin{proof} We refer to \cite{Evans} for the proof of the identity (\ref{divergence_free}). Here, we prove the inequality (\ref{divergence_estimate}). It can be computed from (\ref{optimal_map}) that
\begin{equation}\label{grad_T}
    \nabla T_{u}(x) = -\left(D_{xy}^{2}c(x, T_{u}(x))\right)^{-1}\left(D^{2}u(x) + D_{xx}^{2}c\left(x,T_{u}(x)\right)\right).
\end{equation}
For the cofactor matrix of a matrix $P$, we recall the identity $(\operatorname{det}P)I = P^{\operatorname{T}}(\operatorname{cof}P)$. Therefore, we obtain from (\ref{grad_T}) 
\[\operatorname{cof}\nabla T_{u} = -\operatorname{cof}C\cdot \operatorname{cof}\omega,\]
where $\omega=D^{2}u(x) + A\left(x,\nabla u(x)\right)$ and $C:=D_{xy}^{2}c(x,T_{u}(x))$. Hence,
\[-D_{j}\left((\operatorname{cof}\nabla T_{u})_{ij}\right) = \sum_{k=1}^{n}\,(\operatorname{cof}\omega)_{kj}\cdot D_{j}\left((\operatorname{cof}C)_{ik}^{}\right)+ (\operatorname{cof}C)_{ik}\cdot D_{j}\left((\operatorname{cof}\omega)_{kj}^{}\right).\]
Summing the above formula over $j$, we have the equation
\[0=\sum_{k=0}^{n}(\operatorname{cof}C)_{ik}\left(\sum_{j=1}^{n}D_{j}\left((\operatorname{cof}\omega)_{kj}^{}\right)\right)+\sum_{k,j=1}^{n}(\operatorname{cof}\omega)_{kj}\cdot D_{j}\left((\operatorname{cof}C)_{ik}^{}\right),\]
for $i=1,2,\cdots,n$. It can be reformulated as a linear equation
\[\mathbf{A}\mathbf{x}=\textbf{b},\] 
\[\begin{aligned}
\textbf{x} &= (\text{x}_{1}, \cdots \text{x}_{n})\text{ with } \text{x}_{i} =\sum_{j=1}^{n}D_{j}\left((\operatorname{cof}\omega)_{ij}^{}\right), \qquad \textbf{A} = (C^{\operatorname{T}})^{-1},\\
\textbf{b} &= (\text{b}_{1}, \cdots \text{b}_{n})\text{ with } \text{b}_{i} = -\frac{1}{\operatorname{det} C}\sum_{k,j=1}^{n}(\operatorname{cof}\omega)_{kj}\cdot D_{j}\left((\operatorname{cof}C)_{ik}^{}\right).
\end{aligned}\]
Here, the term $D_{j}\left((\operatorname{cof}C)_{ik}^{}\right)$ in $\textbf{b}$ can be computed by the Jacobi's formula
\[D_{j}\left(\operatorname{cof}C\right) = (C^{\operatorname{T}})^{-1} \Big(\operatorname{tr}\left(\operatorname{cof}C^{\operatorname{T}} \cdot D_{j}C\right)- D_{j}C^{\operatorname{T}}\cdot\operatorname{cof} C\Big).\]
Through a simple computation, we have
\begin{equation}\label{b_i}
\begin{aligned}
    \text{b}_{i} = -\sum_{kj}(\operatorname{cof}\omega)_{kj}\sum_{lmrs}\,\omega_{sj}c^{r,s}&\big(c^{l,i}c^{k,m}c_{m,lr}-\\
    & c^{k,i}c^{m,l}c_{l,mr}\big)+ \left(c^{k,i}c^{m,l}c_{lj,m} - c^{l,i}c^{k,m}c_{mj,l}\right).
\end{aligned}
\end{equation}
The definitions of the notations have been given in (\ref{norm_defn}) and (\ref{deri_defn}). For the elements $\omega_{sj}$ and $(\operatorname{cof}\omega)_{kj}$ in (\ref{b_i}), 
\[\begin{aligned}
\sup_{x\in \mathcal{X}}\sup_{s,j}\left|\omega_{sj}\right|&\leq \left\|\omega\right\|_{2,\mathcal{X}},\\
\sup_{x\in \mathcal{X}}\sup_{k,j}\left|(\operatorname{cof}\omega)_{kj}\right|&\leq \left\|\operatorname{cof}\omega\right\|_{2,\mathcal{X}}.
\end{aligned}\]

Finally, we have the estimate
\[\begin{aligned}
\|\textbf{x}\|_{\infty,\mathcal{X}}\leq \left\|\mathbf{A}^{-1}\right\|_{\infty,\mathcal{X}}\,\|\mathbf{b}\|_{\infty, \mathcal{X}}\leq C\left\|\operatorname{cof}\omega\right\|_{2,\mathcal{X}}\left(\left\|\omega\right\|_{2,\mathcal{X}}+1 \right).
\end{aligned}\]
Here, the constant $C$ depends only on the dimension $n$, the cost function $c$ up to its third-order derivatives and its twist condition,
\[\|c\|_{C^{1\times 2}(\mathcal{X}\times \mathcal{Y})}\, , \quad\|c\|_{C^{2\times 1}(\mathcal{X}\times \mathcal{Y})}\,, \quad\left\|\left(D_{xy}c\right)^{-1}\right\|_{2, \mathcal{X}\times \mathcal{Y}}.\]
Thus, we complete the proof of the inequality (\ref{divergence_estimate}).
\end{proof}

Now we are prepared to evaluate the decreasing of $\mathcal{J}$ on the sequence $\{u^k\}_k$ in (\ref{sobolev_des}).

\begin{proposition}\label{main}
  Suppose that the same conditions as in Theorem \ref{existence} hold, and for sufficiently small $\tau>0$, the sequence $\{u^{k}\}_{k}$ computed by
  \begin{equation}\label{iter_grad}
      \Delta u^{k+1}=\Delta u^{k} +\tau \cdot \mathcal{J}^{\prime}(u^{k}),
  \end{equation}
  is uniformly contained in the neighborhood $V_{\lambda}$ defined by (\ref{V_lambda}). Then there exists $\tau_{max}>0$ such that when $\tau<\tau_{max}$, the functional $\mathcal{J}$ is strictly decreasing on $\{u^{k}\}_{k}$ and 
  \begin{equation}\label{finals}
      \mathcal{J}(u^{k}) - \mathcal{J}(u^{k+1})\geq \frac{1}{2\tau_{max}}\left\|\nabla u^{k}-\nabla u^{k+1}\right\|_{L^2}^{2}.
  \end{equation}
  Here, $\tau_{max}$ depends on the dimension $n$, the parameter $\lambda$ in (\ref{lambda_defn}), the cost function $c$ and the density function $g$.
\end{proposition}

\begin{proof}
From Proposition \ref{dual_grad}, the functional $\mathcal{J}(u)$ is convex and hence for the iteration (\ref{iter_grad}), we have
\begin{equation}\label{des_err}
\begin{aligned}
     &\mathcal{J}(u^{k})-\mathcal{J}(u^{k+1})\geq \int_{\mathcal{X}}\mathcal{J}^{\prime}(u^{k+1})(u^{k}- u^{k+1})\\
     &=\frac{1}{\tau}\int_{\mathcal{X}}(u^{k}- u^{k+1})(\Delta u^{k+1}-\Delta u^{k}) +\int_{\mathcal{X}}(u^{k}- u^{k+1})\left(\mathcal{J}^{\prime}(u^{k+1})-\mathcal{J}^{\prime}(u^{k})\right).
\end{aligned}
\end{equation}
In order to estimate (\ref{des_err}), we first consider 
\begin{equation}\label{full_term}
    \int_{\mathcal{X}}v\cdot\mathcal{L}_{u}v\,\mathrm{d}x = \sum_{i}\int_{\mathcal{X}}v\,\mathcal{L}_{u}^{i}\,D_{i}v+\sum_{ij}\int_{\mathcal{X}}v\,\mathcal{L}_{u}^{ij}\,D_{ij}^2v,
\end{equation}
for any $u\in V_{\lambda}$ and $v\in \tilde{C}^{2}(\mathcal{X})$, where $\mathcal{L}_{u}$ is the linear operator given in Corollary \ref{hessian}. Since $\mathcal{X}$ is a closed manifold without boundary, applying the divergence theorem to the second term in (\ref{full_term}), we obtain
\begin{equation}\label{second_term}
    \begin{aligned}
    \sum_{ij}\int_{\mathcal{X}}v\,\mathcal{L}_{u}^{ij}\,D_{ij}^2v = \sum_{ij}\int_{\mathcal{X}}D_{j}v\,\mathcal{L}_{u}^{ij}\,D_{i}v +\sum_{i}\int_{\mathcal{X}}v \left(\operatorname{div}\mathcal{L}_{u}^{ij}\right) D_{i}v,
\end{aligned}
\end{equation}
where 
\begin{equation}\label{divergence}
    \operatorname{div}\mathcal{L}_{u}^{ij} := \sum_{j}D_{j}(\mathcal{L}_{u}^{ij})=\sum_{j}\tilde{g}(x,\nabla u)D_{j}\left((\operatorname{cof}\omega)_{ij}\right) + \left(\operatorname{cof}\omega\right)_{ij} D_{j}\left(\tilde{g}\left(x,\nabla u\right)\right)
\end{equation}
and 
\[D_{j}\left(\tilde{g}\left(x,\nabla u\right)\right) = \sum_{klrs}D_{k}g\, c^{k,l}\omega_{lj}-g\,c^{s,r}c_{rj,s} + g\,c^{s,r}c_{r,sk}c^{k,l}\omega_{lj}.\]

Using (\ref{divergence_estimate}) in Lemma \ref{prior_estimate}, the divergence (\ref{divergence}) can be bounded by
\[|\operatorname{div}\mathcal{L}_{u}^{ij}| \leq C\|g\|_{C^1}\left\|\operatorname{cof}\omega\right\|_{2,\mathcal{X}}\left(\left\|\omega\right\|_{2,\mathcal{X}}+1 \right),\]
where the constant $C$ depends only on the dimension $n$ and the cost function $c$ as in Lemma \ref{prior_estimate}. 

Similarly, we have $\left|\mathcal{L}_{u}^{i}\right|\leq C \|g\|_{C^1}$ from (\ref{partial_derivatives}). Since $u\in V_{\lambda}$, we obtain $\left\|\omega\right\|_{2,\mathcal{X}}\leq \lambda$ and $\left\|\operatorname{cof}\omega\right\|_{2,\mathcal{X}}\leq \lambda^{n-1}$. Therefore, there exists an upper bound $C$ such that
\begin{equation}\label{bound_1}
    \left|\mathcal{L}_{u}^{i}+\operatorname{div}\mathcal{L}_{u}^{ij}\right|\leq C(n,\lambda,c,g),
\end{equation}
for any $x\in \mathcal{X}$ and $i,j=1\cdots,n$. The constant $C$ depends on the dimension $n$, the parameter $\lambda$, the cost function $c$and the density function $g$.

Hence, combing with (\ref{second_term}), we can estimate (\ref{full_term}) with
\begin{equation}\label{L_estimate}
\begin{aligned}
    \left|\int_{\mathcal{X}}v\cdot\mathcal{L}_{u}v\,\mathrm{d}x\right| &= \left|\sum_{i}\int_{\mathcal{X}}v\left(\mathcal{L}_{u}^{i}+\operatorname{div}\mathcal{L}_{u}^{ij}\right)D_{i}v + \sum_{ij}\int_{\mathcal{X}}D_{i}v\,\mathcal{L}_{u}^{ij}\, D_{j}v\right|\\
    &\leq C\sum_{i}\left(\int_{\mathcal{X}}\left|v\right|\cdot\left|D_{i}v\right|+  \int_{\mathcal{X}}\left|D_{i}v\right|^{2}\right)\\
    & \leq C\left(\left\|v\right\|_{L^{2}}\sum_{i}\left(\int_{\mathcal{X}}|D_{i}v|^{2}\right)^{\frac{1}{2}} + \left\|\nabla v\right\|_{L^{2}}^{2}\right)\\
    &\leq C\left(\left\|v\right\|_{L^{2}}\left\|\nabla v\right\|_{L^{2}} + \left\|\nabla v\right\|_{L^{2}}^{2}\right)\\
    &\leq C\left\|\nabla v\right\|_{L^{2}}^{2},
    \end{aligned}
\end{equation}
where the constant $C$ depends on $n,\lambda,c,g$ as in (\ref{bound_1}). Here, the first inequality in (\ref{L_estimate}) follows from (\ref{bound_1}) and the positive definiteness of the matrix $\mathcal{L}_{u}^{ij}$, which is bounded from above since $u\in V_{\lambda}$. The last inequality in (\ref{L_estimate}) is based on the Poincaré–Wirtinger inequality for $v\in \tilde{C}^{2}(\mathcal{X})$.

For $u^{k},u^{k+1}\in V_{\lambda}$, there exists $s\in (0,1)$ such that
\begin{equation}\label{equal_mean}
    \mathcal{J}^{\prime}(u^{k+1})- \mathcal{J}^{\prime}(u^{k}) = \mathcal{L}_{u}v,
\end{equation}
where $v = u^{k+1}-u^k \in \tilde{C}^{2}$ and $u= su^{k} + (1-s)u^{k+1} \in V_{\lambda}$. 

We can substitute (\ref{equal_mean}) into (\ref{L_estimate}) and obtain
\begin{equation}\label{estimate2}
    \left|\int_{\mathcal{X}}(u^{k+1}-u^{k})\cdot \mathcal{L}_{u}(u^{k+1}-u^{k})\right|\leq C\|\nabla u^{k+1}-\nabla u^{k}\|_{L^2}^{2}.
\end{equation}
By applying the the divergence theorem and the inequality (\ref{estimate2}) to (\ref{des_err}),
\begin{equation}\label{estimate3}
    \mathcal{J}(u^{k})-\mathcal{J}(u^{k+1})\geq \left(\frac{1}{\tau}-C\right)\left\|\nabla u^{k} -  \nabla u^{k+1}\right\|_{L^2}^{2}.
\end{equation}
Choosing the step size $\tau<\tau_{max}:=\frac{1}{2}C^{-1}$ in (\ref{estimate3}), we finish the proof of (\ref{finals}). 
\end{proof}

Next, we study the limit of the sequence $\{u^{k}\}_k\subset V_{\lambda}$ in proposition \ref{main}. Since  the functional $\mathcal{J}$ is bounded from below, by summing the inequality (\ref{finals}), the series
\[\sum_{k=0}^{\infty}\left\|\nabla u^{k} - \nabla u^{k+1}\right\|_{L^2}^{2}<\infty,\]
is convergent and hence $\|\nabla u^{k}-\nabla u^{k+1}\|_{L^2}\rightarrow 0$ as $k\rightarrow \infty$.
For any $c$-convex function $v$, by the convexity of $\mathcal{J}$ and the divergence theorem, 
\[\begin{aligned}
\mathcal{J}(v)-\mathcal{J}(u^{k}) &\geq \int_{\mathcal{X}}\mathcal{J}^{\prime}(u^{k})(v-u^{k})=\frac{1}{\tau}\int_{\mathcal{X}}(\Delta u^{k+1}-\Delta u^{k})(v-u^{k}) \\
&\geq -\frac{1}{\tau}\|\nabla u^{k}-\nabla u^{k+1}\|_{L^2}\|\nabla v-\nabla u^{k}\|_{L^2}.
\end{aligned}\]

By passing the limit $k\rightarrow \infty$, we have $\mathcal{J}(v)\geq \lim_{k\rightarrow\infty} \mathcal{J}(u^{k})$. 
Since $\{u^{k}\}_k$ is uniformly bounded in $V_\lambda$, by the Arzela-Ascoli theorem, there exists a subsequence of $\{u^{k}\}_k$ converges uniformly to some $c$-convex function $\tilde{u}$ under $C^{1,\alpha}$ norm. Hence,
\[\inf_{c\text{-convex}}\mathcal{J}(v)=\lim_{k\rightarrow\infty} \mathcal{J}(u^{k})= \mathcal{J}(\tilde{u}),\]
As shown in Theorem \ref{existence}, the functional $\mathcal{J}$ has a unique $c$-convex minimizer, which means that $\tilde{u}$ is exactly the unique solution $u$ in Theorem \ref{existence}. Since any convergent subsequence of $\{u^{k}\}_k$ converges uniformly to $u$, we obtain the following convergence result.
\begin{theorem}\label{conver}
    Under the same assumptions of Proposition \ref{main}, the sequence $\{u^{k}\}_k$ converges uniformly to $u$ in $\tilde{C}^{1,\alpha}$, $\forall \alpha\in(0,1)$, where $u$ is the solution of the generalized Monge-Ampère equation (\ref{MAO}).
\end{theorem}

It should be noted that Proposition \ref{main} utilizes the smoothness of the cost function $c$ on $\mathcal{X}\times \mathcal{Y}$ to get an upper bound for the step size $\tau$. However, Proposition \ref{main} still holds if the conditions on $c$ are further relaxed. Notice that all the estimates about $c$ in this section can be restricted to the set
\[\operatorname{graph}(T_{u}):=\left\{(x,y):x\in \mathcal{X},\, y= T_{u}(x)\right\},\]
for $c$-convex function $u\in V_{\lambda}$. Therefore, the domain of $c$ can be reduced from the original $\mathcal{X}\times \mathcal{Y}$ to 
\[\operatorname{G}_{\lambda}:=\underset{u\in V_{\lambda}}{\bigcup} \operatorname{graph}\left(T_{u}\right),\]
i.e., the cost function $c$ only needs to be smooth on $\operatorname{G}_{\lambda}$.

For the far-field reflector problem, the cost function $c(x,y)= -\log(1-x\cdot y)$ on $\mathbb{S}^{2}$ fails to be smooth on the singularity set 
\[\operatorname{sing}(c):=\left\{(x,y)\in \mathbb{S}^{2}\times \mathbb{S}^{2}: x=-y\right\}.\]
For this cost function, \cite{loeper2011regularity} demonstrates that there exists some constant $\delta>0$ such that
\begin{equation}\label{sing}
    d\left(\operatorname{sing}(c),\operatorname{G}_{\lambda}\right)>\delta,
\end{equation}
where $d$ is the metric on $\mathbb{S}^{2}$ and $\delta$ depends only on the parameter $\lambda$. The formula (\ref{sing}) means that the reflector cost $c(x,y)= \log(1-x\cdot y)$ is smooth on $\operatorname{G}_{\lambda}$. Hence, both Proposition \ref{main} and Theorem \ref{conver} can be extended to the far-field reflector problem.

\begin{remark}
The results in Section 3 and Section 4 are discussed for the entire space $\mathbb{S}^{n}$ or $\mathbb{T}^{n}$. For the far-field reflector problem, the density functions $f,g$ are supported on $\Omega\subset \mathbb{S}^{-}$ and $\Omega^*\subset \mathbb{S}^{-}$, respectively. To reduce the computation, we may solve (\ref{sobolev_des}) on $\Omega$ instead. An OT boundary condition \cite{benamou2014numerical}
\begin{equation}\label{bd}
    T_{u^{k+1}}(\partial \Omega) = \partial \Omega^{*},
\end{equation}
should be applied to the descent iteration (\ref{sobolev_des}). The $c$-convexity of $u^{k+1}$ in terms of the space $\Omega\times\Omega^{*}$ can be ensured by the condition (\ref{c_convex}) together with (\ref{bd}), which leads to a diffeomorphism $T_{u^{k+1}}$ from $\Omega$ to $\Omega^*$. It should be noted that the $c$-convexity of $u^{k+1}$ is important for the descent scheme (\ref{sobolev_des}), since $\mathcal{J}^{\prime}(u^{k+1})$ can not give a descent direction for $\mathcal{J}$ if $u^{k+1}$ is not $c$-convex.
\end{remark}

\section{Numerical method for the freeform reflector design}
In this section, based on the iteration (\ref{sobolev_des}), we employ the mixed finite element scheme in \cite{mcrae2018optimal} to solve the far-field reflector problem. The numerical experiments are performed using the open-source finite element framework Firedrake.

We use the notation $\mathcal{M}_{h}$ to denote the triangular partition of the computational domain, $V_h$ to denote the Lagrange finite element space consisting of continuous piecewise polynomials on $\mathcal{M}_{h}$, and $\Sigma_{h}:=(V_{h})_{3\times 3}$. For simplicity, $\left\langle\cdot,\cdot\right\rangle_{\Omega}$ is used to denote the $L^2$ inner product on $\Omega$.

The far-field design is equivalent to the optimal transport problem with $\mathcal{X}=\mathcal{Y}=\mathbb{S}^{2}$ and $c(x,y) = -\log(1-x\cdot y)$. The iterative scheme (\ref{sobolev_des}) for solving OT is a Poisson equation with respect to $u^{k+1}$, whose weak formulation is written as
\begin{equation}\label{w_S}
\left\langle\nabla v, \nabla u^{k+1}\right\rangle=\left\langle\nabla v, \nabla u^k\right\rangle+\tau\left\langle v, r^k\right\rangle, \quad \forall v \in V_h.
\end{equation}
Here, $r^{k}$ is the residual of the generalized Monge-Ampère equation,
\begin{equation}\label{res}
    r^{k}(x):= g\big(T_{u^k}(x)\big) \left|\operatorname{det}\big(\nabla  T_{u^k}(x)\big)\right|-f^k(x),
\end{equation}
where $T_{u}$ is the map defined by (\ref{optimal_map}) and $f^{k}(x):=\theta^{k}f(x)$,
\[\theta^{k}:= \left(\int_{\mathbb{S}^2}f\right)^{-1}\cdot\left(\int_{\mathbb{S}^2}g\big(T_{u^k}\big) \left|\operatorname{det}\big(\nabla  T_{u^k}\big)\right|\right).\]

For the far-field reflector problem, $T_{u}$ is exactly the optical map in (\ref{reflector_map3}). However, when applying the finite element method, the computational domain $\mathcal{M}_{h}$ is embedded in $\mathbb{R}^{3}$. The Jacobian matrix computed here is actually $\nabla_{\mathbb{R}^{3}}T_{u}$ instead $\nabla_{\mathbb{S}^{2}}T_{u}$ that we need. Hence, we recompute the Jacobian determinant in (\ref{res}) using the basic definition
\begin{equation}\label{JO}
    \left|\operatorname{det}\big(\nabla_{\mathbb{S}^{2}} T_{u}(x)\big)\right| = \lim_{|U|\rightarrow 0}\frac{|T_{u}(U_{x})|}{|U_{x}|},\quad U_{x} \text{ is the neighbourhood of } x.
\end{equation}

The computation of the Jacobian determinant follows the strategies of \cite{mcrae2018optimal}. For $x\in\mathbb{S}^{2}$, we parameterize its tangent plane using orthogonal vectors $\{e_1, e_2\}$. The corresponding surface area element $U$ equals to $|\operatorname{det}(e_1,e_2,x)|$. Accordingly, the surface area element of $|T_{u}(U)|$ is $|\operatorname{det}(v_1, v_2,y)|$, where $v_i = \nabla_{\mathbb{R}^3} T_{u}(x)\,e_i$, $i=1,2$ and $y = T_{u}(x)$. Constructing the projection matrix,
\begin{equation}
\left\{\begin{aligned}
    P(x)&:=I-xx^{\operatorname{T}},\\
    Q_{u}(x)&:=T_{u}(x)\,x^{\operatorname{T}},
\end{aligned}\right.
\end{equation}
we obtain
\[(v_1, v_2,y)= \Big(\nabla_{\mathbb{R}^3} T_{u}(x)P(x) + Q_u(x)\Big)(e_1,e_2,x).\]
The Jacobian determinant of $T_{u}:\mathbb{S}^{2}\rightarrow \mathbb{S}^{2}$ for $c$-convex $u$ can be given by
\begin{equation}
    \left|\operatorname{det}\big(\nabla_{\mathbb{S}^{2}} T_{u}(x)\big)\right| = \operatorname{det}\Big(\nabla_{\mathbb{R}^3} T_{u}(x)\,(I-xx^{\operatorname{T}})+ T_{u}(x)\,x^{\operatorname{T}}\Big),
\end{equation}
which makes it easy to compute $r^k$ in (\ref{res}) under the finite element framework.

In practice, we choose the computational domain as the entire sphere if intensity distributions $f,g$ are supported on complicated subsets of $\mathbb{S}^{2}$. On the other hand, if the structures of  $\Omega:=\operatorname{supp}f\subset\mathbb{S}_{-}^{2}$ and $\Omega^{*}:=\operatorname{supp}g\subset \mathbb{S}_{+}^{2}$ are regular, the computational domain can be chosen as $\Omega$ to improve the efficiency and accuracy.

In this case, the boundary condition (\ref{bd}) should be applied to Possion's equation (\ref{sobolev_des}). Here we use the Neumann boundary condition to realize (\ref{bd}) instead, which means that a priori information of $\nabla u^{k+1}\cdot \nu$ is needed, where $\nu$ is the normal of $\partial \Omega$.

Suppose that we already obtain $u^k$ at the $k$-th step. Then $u^k$ defines the corresponding map $T_{u^{k}}$, which is expected to satisfy the boundary condition $T_{u^{k}}(\partial \Omega)=\partial \Omega^{*}$. To achieve this, we project each point $y \in T_{u^{k}}(\partial \Omega)$ onto the closest point $p$ on $\partial \Omega^*$, 
\begin{equation}\label{proj}
    p^{k}(x):=\operatorname{Proj}_{\partial \Omega^{*}}\left(T_{u^{k}}^{}(x)\right),\quad x\in \partial \Omega,
\end{equation}
where
\[\operatorname{Proj}_{\partial \Omega^{*}}(y):= \operatorname{exp}_{y}\left(H(y)\nabla H(y)\right),\quad y\in\mathbb{S}^{2}.\]
Here, the symbol $\exp$ denotes the exponential map on the tangent bundle of $\mathbb{S}^{2}$, and $H(y)$ is the distance function of $\partial \Omega^*$,
\begin{equation}\label{dist}
H(y):= \begin{cases}-\operatorname{dist}(y, \partial \Omega^*), & y \in \Omega^*, \\ +\operatorname{dist}(y, \partial \Omega^*), & y \in \mathbb{S}^{2} /\Omega^*.\end{cases}
\end{equation}
We can update the information of $\nabla u^{k}$ from $p^{k}(x)$ in (\ref{proj}) and denote it as
\begin{equation}\label{Neu}
h^{k+1}(x):= \mathcal{T}^{-1}\left(x,\, p^{k}(x)\right)=-\nabla_{x} c\left(x, p^{k}(x)\right),\quad x\in \partial\Omega.
\end{equation}
For the far-field reflector problem $c(x,y)=-\log(1-x\cdot y)$, it is equal to
\[h^{k+1}(x)=\frac{p^k(x)-\left(x\cdot p^k(x)\right)x}{1-x\cdot p^k(x)},\quad x\in \partial\Omega.\]
The vector function $h^{k+1}$ can be viewed as the information of $\nabla u^{k+1}$ on $\partial \Omega$.

In order to compute the Jacobian determinant precisely, $\nabla_{\mathbb{R}^3} T_{u^k}$ is represented by a tensor-valued variable $\sigma^k\in\Sigma_{h}$, which solves the linear variational formulation,
\begin{equation}\label{mixed}
    \left\langle\sigma^k, \chi\right\rangle_{\Omega}=-\left\langle\operatorname{div} \chi, T_{u^k}\right\rangle_{\Omega} +\left\langle T_{u^k}, \chi \nu\right\rangle_{\partial \Omega}, \quad \forall \chi \in \Sigma_h.
\end{equation}
Therefore, the weak form of (\ref{sobolev_des}) on $\Omega$ can be formulated as
\begin{equation}\label{weak_p}
    \left\langle \nabla u^{k+1}, \,\nabla v\right\rangle_{\Omega} = \left\langle \nabla u^{k}, \,\nabla v\right\rangle_{\Omega} + \tau\left\langle r^{k},\,v\right\rangle_{\Omega}+ \left\langle\nu,v(h^{k} - h^{k+1})\right\rangle_{\partial \Omega},\quad v\in V_{h}.
\end{equation}
Here, $r^{k}$ defined by (\ref{res}) is rewritten as 
\[\begin{aligned}
r^{k}(x) &= g\big(T_{u^k}(x)\big)\operatorname{det}\Big(\sigma^{k}(x)(I-xx^{\operatorname{T}})+ T_{u^k}(x)\,x^{\operatorname{T}}\Big)-f^k(x),\\
f^{k}(x) &:= \theta^{k}\cdot f(x),
\end{aligned}\]
where 
\[\theta^{k}:=\left(\int_{\Omega}f\right)^{-1}\cdot\left(\int_{\Omega}g(T_{u^k})\operatorname{det}\big(\sigma^{k}(I-xx^{\operatorname{T}})+ T_{u^k}\,x^{\operatorname{T}}\big)\right).\]

The iteration should be terminated when the functional $\mathcal{J}(u^k)$ stops decreasing. However, it is complicated to compute $\mathcal{J}(u^{k})$ due to the computation of $u^{c}(y)$. Instead, we use the $L^2$ norm of $r^{k}$,
\[\|r^{k}\|_2:=\left(\int_{\Omega}|r^k(x)|^2\,\mathrm{d}x\right)^{2},\]
as the stop criterion. The numerical method for solving the far-field reflector problem on $\Omega$ is summarized in Algorithm \ref{alg1}.

\begin{algorithm}
    \caption{FEM for the far-field reflector design on $\Omega$.}\label{FF}
    \label{alg1}  
    \begin{algorithmic}  
    \STATE Given the initial value $u^{0}$ and step size $\tau$.
    \WHILE{$\|r^{k-1}\|_{2}< \|r^{k-2}\|_{2}$}
    \STATE Solving the linear variational problem (\ref{mixed}) to obtain $\sigma^k$.
    \STATE Evaluating the composite function $g(T_{u^k})$.
    \STATE Solving the weak formulation of Poisson’s equation (\ref{weak_p}) to obtain $u^{k+1}$.
    \STATE Computing the error $\|r^k\|_{2}$. 
    \STATE $k:= k+1$.
    \ENDWHILE
    \end{algorithmic}  
\end{algorithm}

\section{Numerical experiments}
In this section, we present the numerical experiments for the far-field reflector problem using the algorithm in Section 5. We define the subdomain $\mathbb{S}_{-\theta}^{2}\subset \mathbb{S}^{-}$ as
\[\mathbb{S}_{-\theta}^{2}:=\left\{x\in \mathbb{S}^{2}:  x\cdot e_{\boldsymbol{z}} \leq -\cos(\theta), \;\text{where } e_{\boldsymbol{z}} = (0,0,1)^{\operatorname{T}}\right\},\quad \theta \in \left(0,\frac{\pi}{2}\right),\]
and $\mathbb{S}_{+\theta}^{2}\subset\mathbb{S}^{+}$ as
\[\mathbb{S}_{+\theta}^{2}:=\left\{x\in \mathbb{S}^{2}:  x\cdot e_{\boldsymbol{z}} \geq \cos(\theta), \;\text{where } e_{\boldsymbol{z}} = (0,0,1)^{\operatorname{T}}\right\},\quad \theta \in \left(0,\frac{\pi}{2}\right).\]

We present several numerical experiments by designing different target intensities $g$. The source intensity 
\begin{equation}\label{source}
    f(x)=\left\{\begin{aligned}1,\qquad &x\in\mathbb{S}_{-\pi/4}^{2},\\
    0,\qquad &\text{else,}\end{aligned}\right.
\end{equation}
supported on $\Omega = \mathbb{S}_{-\pi/4}^{2}$ is fixed. In fact, the computational stability of the numerical method depends more on the target $g$ than the source $f$. In experiments, $\mathcal{M}_{h}$ is the second-order mesh discretized on $\Omega = \mathbb{S}_{-\pi/4}^{2}$. The finite element space $V_h$ is of second order. Linear equations in Algorithm \ref{alg1} are solved by incomplete LU factorization preconditioned CG method, where the constraint $\int_{\Omega}u^{k+1}=0$ is handled by the constant null space of the Krylov solver.

In this section, $\|r\|_{2}:= \|r(x)\|_{L^2}$ is used to denote the $L^2$ norm of the residual. The illumination is simulated using the ray-tracing method. We plot the ray-traced image ${T_{u_{h}}}_{\,\#} f$ using the kernel density function in Matlab, where $u_h$ is the numerical solution of the problem, and the notation $\#$ is defined in (\ref{primal}). To ensure the convergence of the method, the initial value $u^0$ should be a $c$-convex function, which is generally zero. If the target $g$ is far away from ${T_{u^{0}}}_{\,\#} f$, we may choose a middle value $g_{mid}$ and solve the reflector problem with the target $g_{mid}$. Then, the solution of this problem can be used as the initial value for the reflector problem with the target $g$.

\begin{example} \textbf{Smooth target.} In this example, we design a freeform reflector to produce the smooth target intensity $g$ in the far-field. Here, the intensity $g$ is compactly support on $\mathbb{S}_{+\pi /4}^{2}$, shown Figure \ref{ex1r}(a). The source intensity $f$ is given by (\ref{source}). The initial function $u^0$ in Algorithm \ref{alg1} is set to be zero. We choose the step size $\tau =0.5$ and the finite element mesh size $h = 9.82\times 10^{-3}$.
\begin{figure*}[h]
    \centering
    \includegraphics[width=0.8\textwidth]{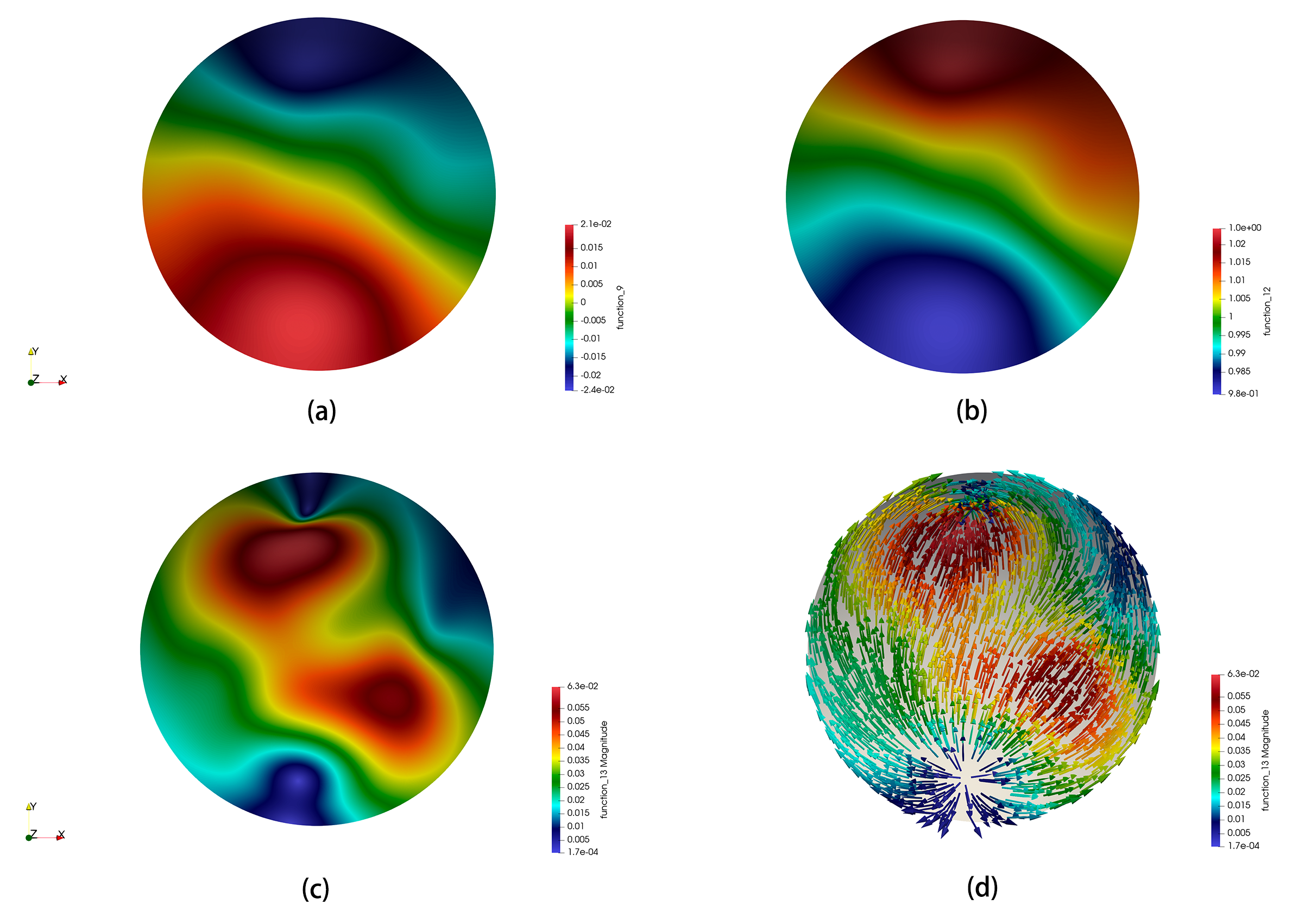}
    \caption{Results of Example 6.1. (a) Numerical solution $u_{h}$ on $\mathbb{S}_{-\pi/4}^{2}$, where the mesh size $h=9.82\times 10^{-3}$. (b) Radial distance function $\rho_h := e^{-u_h}$. (c) Magnitude $|\nabla \rho_{h}|$. (d) Vector field $\nabla \rho_{h}$, where the colors correspond to the values of $|\nabla \rho_{h}|$.} \label{ex1}
\end{figure*}

\begin{figure*}[h]
    \centering
    \includegraphics[width=1\textwidth]{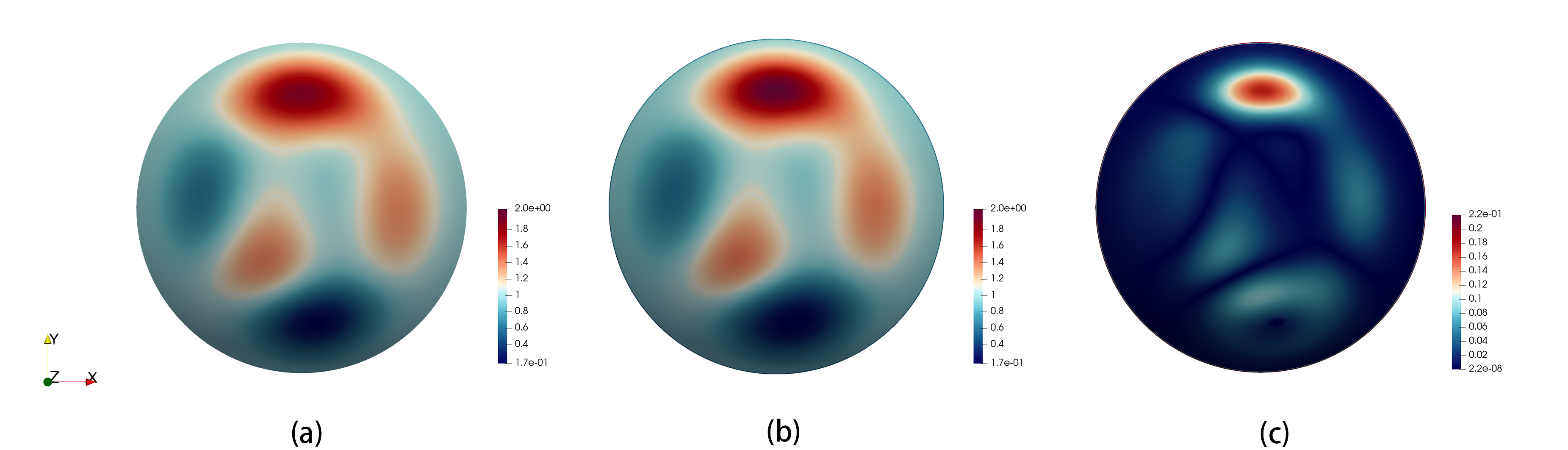}
    \caption{Ray-traced image of Example 6.1. (a) Far-field target intensity $g$ on $\mathbb{S}_{+\pi /4}^{2}$. (b) Far-field ray-traced intensity $g_{h}$ using the numerical solution $u_{h}$, where $h=9.82\times 10^{-3}$. (c) Absolute error $|g-g_{h}|$.} \label{ex1r}
\end{figure*}

Figure \ref{ex1}(a) and (b) demonstrate the numerical solution $u_h$ and its gradient $\nabla u_h$, respectively. In Figure \ref{ex1r}(b), we presents the ray-traced result $g_h:= {T_{u_h}}_{\,\#}f$ using the computed $u_{h}$ in Figure \ref{ex1}(a). It is difficult to distinguish the difference between $g$ and $g_h$ from Figure \ref{ex1r}(a) and (b). Further, we plot the absolute error $|g-g_h|$ in Figure \ref{ex1r}(c) to evaluate the difference between $g$ and $g_h$. The error is relatively small in most areas, and the maximal absolute error 0.2 occurs near the peak of the intensity $g$. Although the intensity $g$ is smooth, the high contrast of $\frac{\max g}{\min g}\approx 12$ makes the condition number of $\nabla T_{u_{h}}$ small, decreasing the accuracy of the numerical solution.

\begin{figure*}[h]
    \centering
    \includegraphics[width=0.43\textwidth]{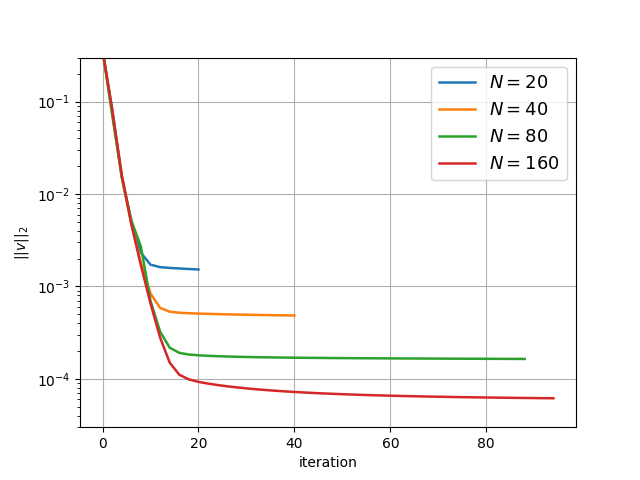}
    \caption{Convergence curves of Example 6.1 at different discretization levels, residual $\|r\|_{2}$ versus iterations.}\label{ex1c}
\end{figure*}

Figure \ref{ex1c} shows the changes of residual value $\|r\|_{2}$ over iterations for mesh levels $N =20,\,40,\,80,\,160$, which correspond to the mesh sizes $h = 3.93\times 10^{-2}$, $1.96\times 10^{-2}$, $9.82\times 10^{-3}$, $4.9\times 10^{-3}$. Here $N$ refers to the number of mesh cells along the geodesic radius of $\mathbb{S}_{-\pi/4}^{2}$. The algorithms for $N =20,\,40,\,80,\,160$ stop after 11, 21, 45, and 47 iterations, respectively.

\end{example}

\begin{example}\textbf{Off-axis target}. In this example, the far-field target intensity $g$ is 
\begin{equation}\label{G2}
    g(x)=\left\{\begin{aligned}
    1,\qquad &\text{if }  x\cdot q \geq \cos\left(\frac{\pi}{4}\right), \text{ where } q =\left(0, -\sin\left(\frac{\pi}{8}\right), \cos\left(\frac{\pi}{8}\right)\right)^{\operatorname{T}},\\
    0,\qquad &\text{else},
\end{aligned}\right.
\end{equation}
which corresponds to the blue region in Figure \ref{ex2}(a). The source target $f$ defaults to (\ref{source}), corresponding to the orange region in Figure \ref{ex2}(a). The step size $\tau = 0.5$ and the mesh size $h=9.82\times 10^{-3}$.
\begin{figure*}[h]
    \centering 
    \includegraphics[width=1\textwidth]{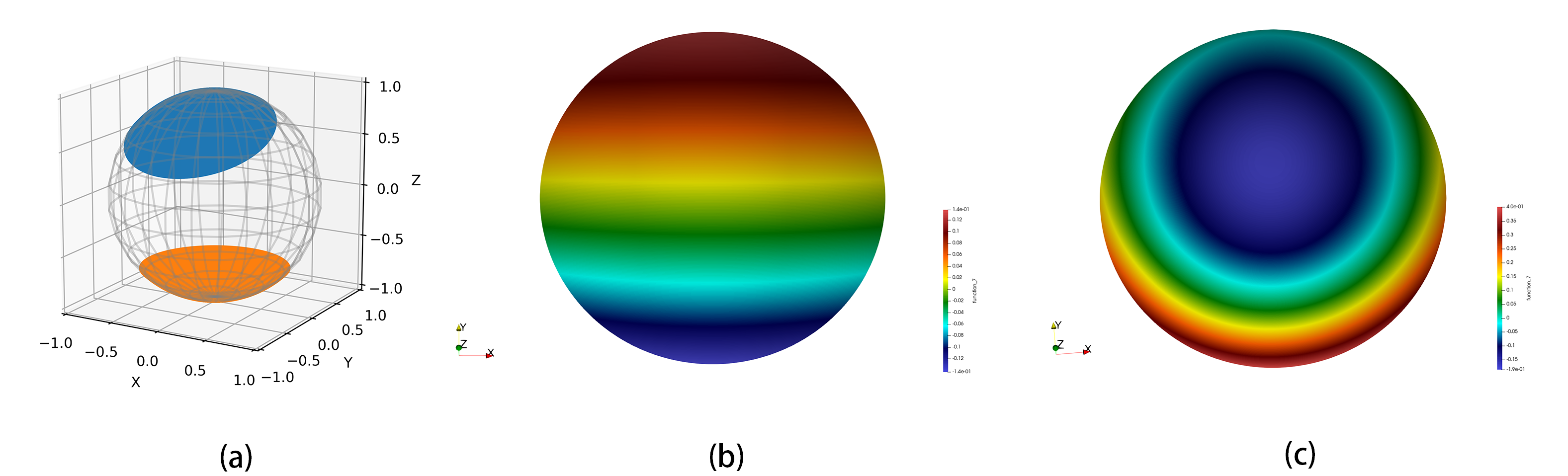}
    \caption{Results of Example 6.2. (a) The source intensity $f$ (orange) and the far-field target intensity $g$ (blue) on $\mathbb{S}^{2}$. (b) Numerical solution $u_{h}$ on $\mathbb{S}_{-\pi /4}^{2}$, where the mesh size $h=9.82\times 10^{-3}$. (c) Numerical solution $\varphi_{h}$ with $h=9.82\times 10^{-3}$. }\label{ex2}
\end{figure*}

\begin{figure*}[h]
    \centering 
    \includegraphics[width=1\textwidth]{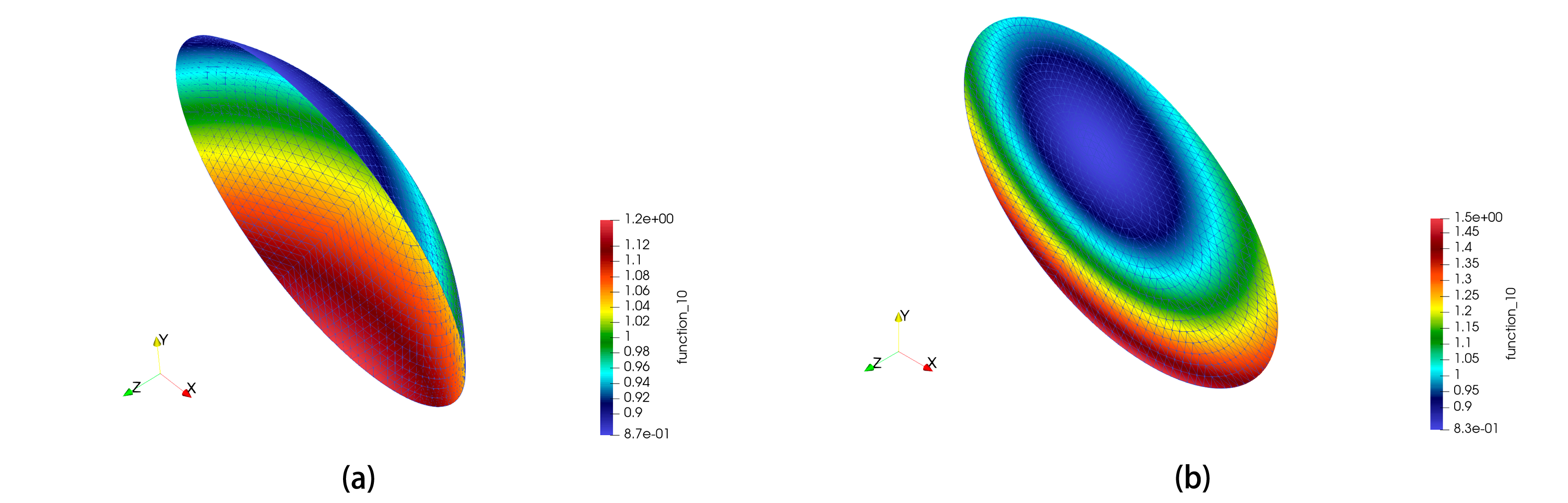}
    \caption{Results of Example 6.2. (a) The reflecting surface $\Gamma =\{x \rho_h(x),\;x \in \mathbb{S}_{-\pi /4}^{2} \}$, where $\rho_h= e^{-u_h}$. (b) The reflecting surface $\Gamma =\{x \rho_h(x),\;x \in \mathbb{S}_{-\pi /4}^{2}\}$, where $\rho_h= e^{\varphi_h}$. The colors of the surfaces correspond to the values of $\rho_h$.}\label{ex2f}
\end{figure*}

In this example, we solve the far-field problem using two different formulations mentioned in Theorem \ref{OT_linkA} and Theorem \ref{OT_linkB}, i.e., the optimal transport problem with the cost function $c(x,y) = -\log(1-x\cdot y)$ and $c(x,y)=\log(1-x\cdot y)$. Here, their numerical solutions are denoted by $u_{h}$ and $\varphi_h$,  which are shown in Figure \ref{ex2} (b) and (c), respectively. The corresponding reflecting surfaces are shown in Figure \ref{ex2f}.

In experiments, we initialize 
\[\begin{aligned}
    u^0&=0,\\
    \varphi^0& =\log\left(\frac{b}{x\cdot e_{\boldsymbol{z}}}\right),\;e_{\boldsymbol{z}} =(0,0,1)^{\operatorname{T}},
\end{aligned}\]
and $b$ is some negative constant such that $\int_{\mathbb{S}_{-\pi/4}^{2}}\varphi^0 =0$. In particular, ${T_{u^0}}_{\,\#} f = {T_{-\varphi^0}}_{\,\#} f=f(-x)$, i.e., the uniform distribution on $\mathbb{S}_{+\pi/4}^{2}$. The center of $f(-x)$ forms an angle of $\frac{\pi}{8}$ with the center of $g$, increasing the difficulty of the boundary match. 
\begin{figure*}[h]
    \centering
    \includegraphics[width=0.85\textwidth]{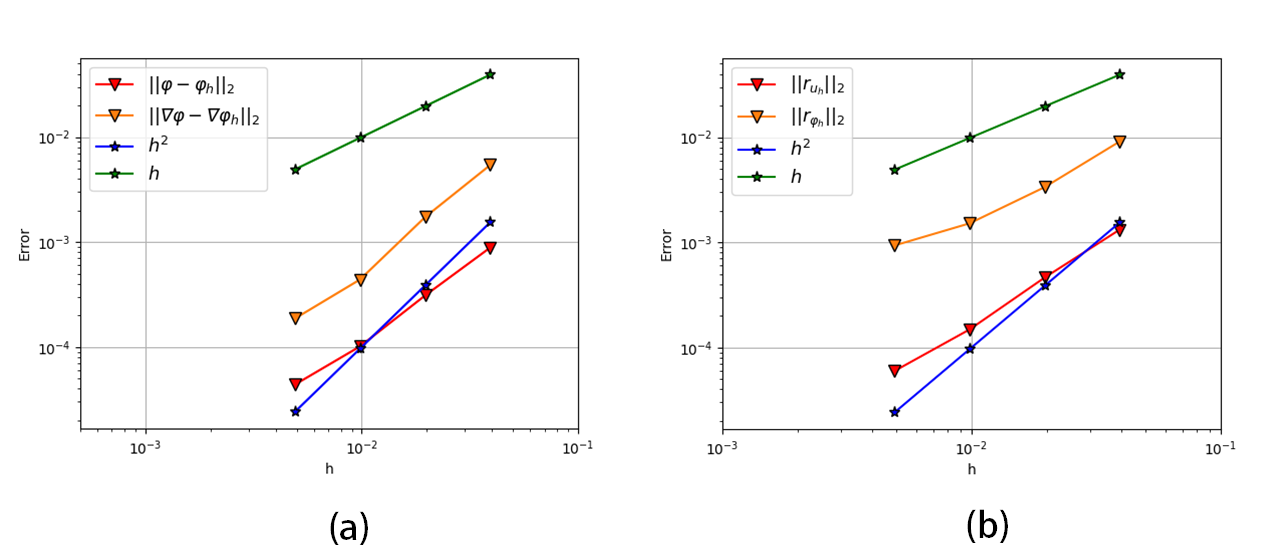}
    \caption{Error plots for the numerical solutions $u_h, \varphi_h$ of Example 6.2. (a) $\|\varphi-\varphi_{h}\|_{2}$ and $\|\nabla\varphi-\nabla\varphi_{h}\|_{2}$ versus the mesh size $h$, where $\varphi$ is the exact solution. (b) Residual values $\|r_{u_h}\|_{2}$ and  $\|r_{\varphi_h}\|_{2}$ versus $h$, where $r_{u_h}$ and $r_{\varphi_h}$ are the residuals of $u_{h}$ and $\varphi_{h}$, respectively.}\label{ex2e}
\end{figure*}

We remarked that for this example, OT with the cost $c(x,y)=\log(1-x\cdot y)$ has the exact solution $\varphi$ 
\[\varphi(x) = \log\left(\frac{b}{x\cdot p}\right),\qquad p = \left(\frac{q_1}{\sqrt{2(q_3+1)}}, \frac{q_2}{\sqrt{2(q_3+1)}},\sqrt{\frac{q_{3}+1}{2}}\right)^{\operatorname{T}},\quad x\in\mathbb{S}_{-\pi/4}^{2},\]
where the vector $q=(q_1,q_2,q_3)^{\operatorname{T}}$ is given in (\ref{G2}) and $b$ is an arbitrary negative constant, which is often taken to make $\int_{\mathbb{S}_{-\pi/4}^{2}}\varphi =0$. The exact reflector surface of this case is a plane in $\mathbb{R}^{3}$, as shown in Figure \ref{ex2f}(b). We plot the errors $\|\varphi-\varphi_{h}\|_{2}$ and $\|\nabla\varphi-\nabla\varphi_{h}\|_{2}$ versus the mesh size $h$ in Figure \ref{ex2e}(a). Meanwhile, the changes of the residual values $\|r_{u_h}\|_{2}$ and $\|r_{\varphi_h}\|_{2}$ with respect to the mesh size $h$ have been shown in Figure \ref{ex2e}(b). It is obvious that the solution $u_{h}$ has higher accuracy than $\varphi_h$.

\end{example}

\begin{example}\textbf{Discontinuous target with singular background}.
Different from the previous examples, we set the far-field illumination area to be a far-field plane $P_{\infty}$ instead of a far-field sphere. In addition, the target intensity $g$ is supported on the square area of $P_{\infty}$, which has an image of the letter "A" as shown in Figure \ref{near}(a). For the visualization of the image on the plane, $P_{0.5}:=\{x=(x_1,x_2,x_3)^{\operatorname{T}}\in \mathbb{R}^{3}:x_3=0.5\}$ is used to substitute $P_{\infty}$. We use $T_{u}^{\operatorname{P}}$ to denote the reflecting map from the source $\mathbb{S}^2$ to the target plane $P_{0.5}$.

Utilizing the stereographic projection, we can project $g$ on $P_{0.5}$ onto a new function $\hat{g}$ on $\mathbb{S}^{2}$, where $\hat{g}$ defined by $\hat{g}:= \mathcal{S}g$,
\[\begin{aligned}\mathcal{S}:L^{1}(P_{0.5})&\rightarrow L^{1}(\mathbb{S}_{-\pi/4}^{2}),\\
g&\rightarrow \hat{g}, \qquad \qquad\text{where}
\end{aligned}\]
\[\hat{g}(x) = \frac{\left( X^2 +Y^2+ 0.5^2\right)^{1.5}}{0.5} g\left(X,Y\right), \quad x = (x_1,x_2,x_3)^{\operatorname{T}}\in \mathbb{S}_{-\pi/4}^{2},\]
\[\text{where}\qquad\left\{\begin{aligned}
X = 0.5\frac{\sqrt{x_1^2 +x_2^2}}{x_3} \cos\left(\tan^{-1}\left(\frac{x_2}{x_1}\right)\right),\\
Y = 0.5\frac{\sqrt{x_1^2 +x_2^2}}{x_3} \sin\left(\tan^{-1}\left(\frac{x_2}{x_1}\right)\right).
\end{aligned}\right.\] 
Figure \ref{ex4e}(a) shows the image of $\hat{g}$. We can solve the far-field reflector problem for the source $f$ in (\ref{source}) and the target $\hat{g}$. The obtained solution $u$ can produce the illumination pattern of Figure \ref{near}(a) on the far-field plane, and the illumination $T_{u\; \#}^{\operatorname{P}}f$ is related to $T_{u\, \#}f$ by $T_{u\; \#}^{\operatorname{P}}f=\mathcal{S}^{-1}(T_{u\,\#}f)$.

Both the shape of $\operatorname{supp}g$ and the discontinuity of $g$ increase the difficulty of solving for $u$ efficiently. We set the step size $\tau = 0.3$ and the initial value $u^0=0$. Due to the discontinuity of the target $g$, the algorithm stops after 12 iterations for $h = 9.82\times 10^{-3}$, with the residual value $r_{u_h}$ equal to $4\times 10^{-2}$. 

\begin{figure*}[h]
    \centering
    \includegraphics[width=1\textwidth]{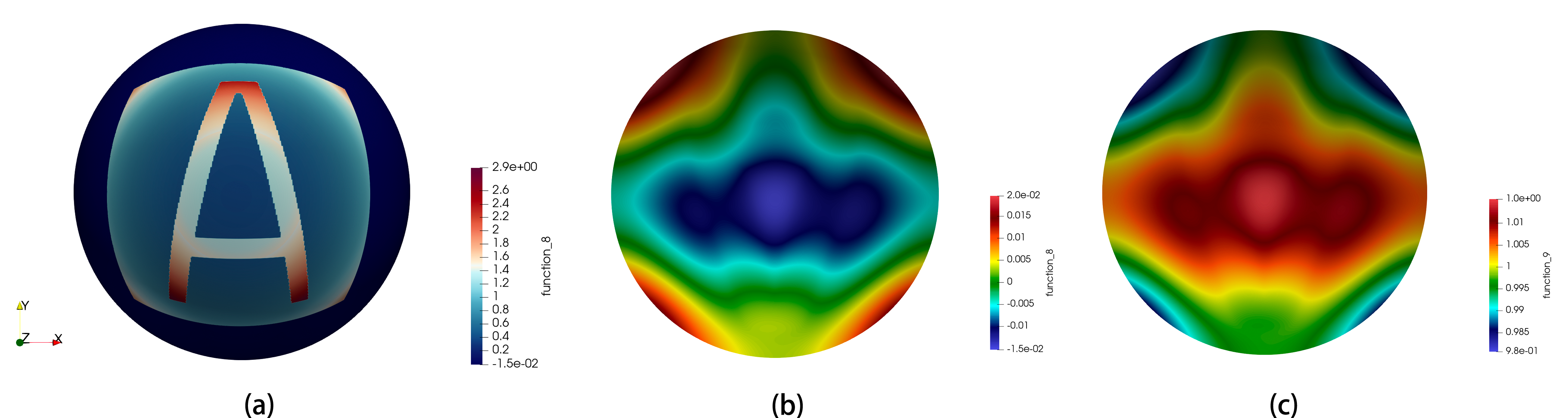}
    \caption{(a) Far-field target intensity function $\hat{g}$ on $\mathbb{S}_{+}^{2}$, (b) numerical solution $u_h$ on $\mathbb{S}_{-\pi/4}^{2}$ for $h =9.82\times 10^{-3}$ and (c) radial distance function $\rho_h := e^{-u_h}$.}\label{ex4e}
\end{figure*}

\begin{figure*}[h]
    \centering
    \includegraphics[width=0.63\textwidth]{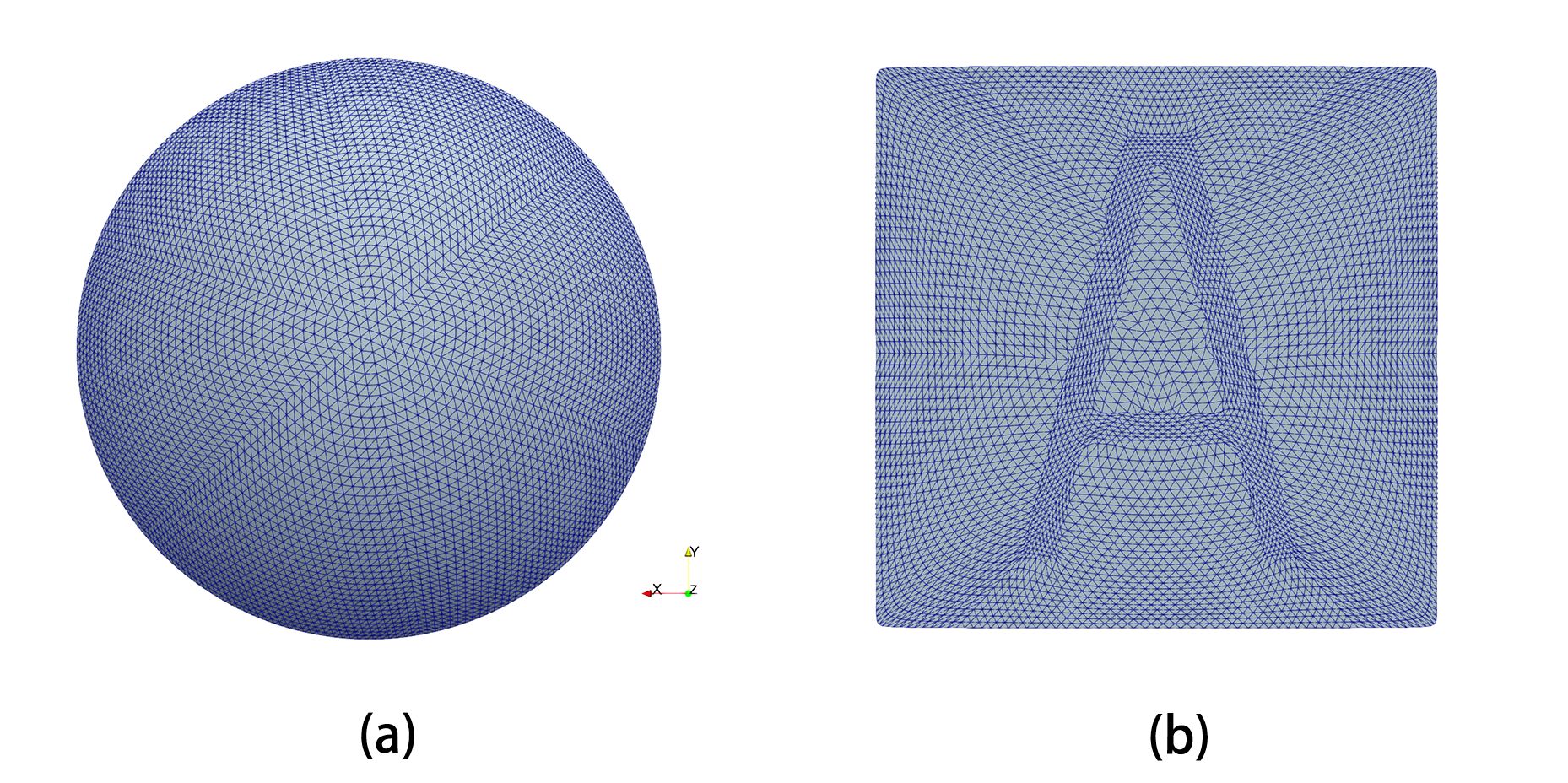}
    \caption{Results of Example 6.3. (a) Finite element mesh $\mathcal{M}_{h}$ with $h=1.96\times 10^{-2}$. (b) New mesh obtained by the far-field mapping $T_{u_h}^{\operatorname{P}}\mathcal{M}_{h}$, $h=1.96\times 10^{-2}$. }\label{nearsol}
\end{figure*}
Figure \ref{ex4e}(b) and (c) show the numerical solution $u_h$ of the far-field problem and its radial distance function $\rho_h$, where $h = 9.82\times 10^{-3}$. Figure \ref{nearsol}(a) and (b) respectively show the finite element mesh $\mathcal{M}_{h}$ discretizing $\mathbb{S}_{-\pi/4}^{2}$ and the mesh obtianed by $T_{u_h}^{\operatorname{P}}\mathcal{M}_{h}$, where $h=1.96\times 10^{-2}$.
\begin{figure*}[h]
    \centering
    \includegraphics[width=1\textwidth]{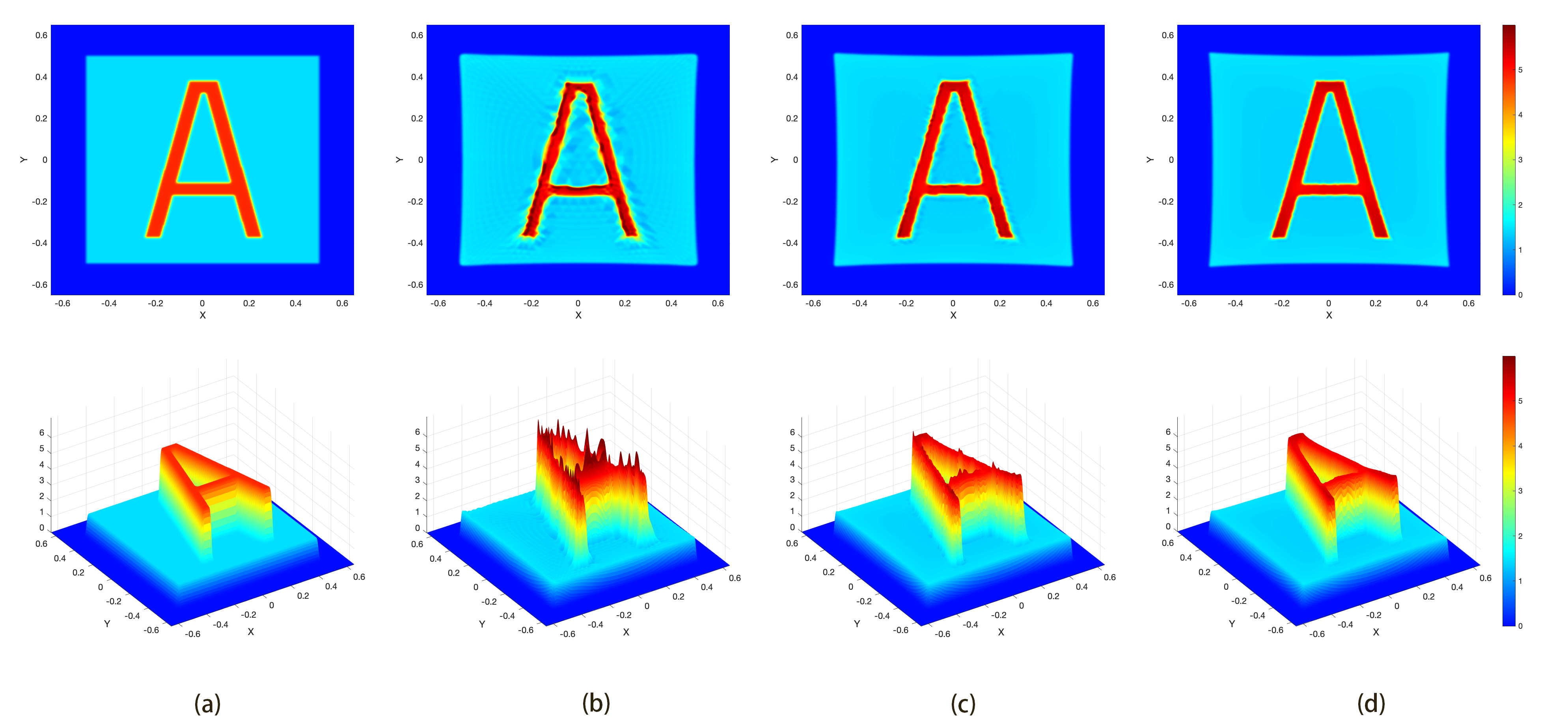}
    \caption{Ray-traced images $g_{h}:= T_{u_h\; \#}^{\operatorname{P}}f$ on the plane, using the results $u_h$ of Example 6.3. From left to right: (a) target intensity $g$ on $P_{0.5}$, (b) ray-traced intensity $g_{h}$ with $h= 3.93×\times 10^{-2}$, (c) ray-traced intensity $g_{h}$ with $h= 1.96\times 10^{-2}$, (d) ray-traced intensity $g_{h}$ with $h= 9.82\times 10^{-3}$. The first row is the image on the 2D plane, and the second row is its 3D view.}\label{near} 
\end{figure*}

In Figure \ref{near}, the far-field ray-traced images $g_h:= T_{u_h\; \#}^{\operatorname{P}} f$ on the target plane $P_{0.5}$ are shown for the mesh sizes (b) $h=3.93×\times 10^{-2}$, (c) $h=1.96\times 10^{-2}$, (d) $h=9.82\times 10^{-3}$. The ray-traced illumination on the target plane is increasingly clear as the mesh is refined.

\end{example}

\section{Conclusions}

This work introduces a fast algorithm derived from the optimal transportation theory to address the far-field reflector problem. The method employs the Sobolev gradient descent on the reduced objective functional of OT, achieving a balance between computational stability and efficiency. Theoretical analysis verifies the superior convergence properties of the descent scheme, while experimental results demonstrate that the proposed method performs exceptionally well even in singular cases.

Future work may focus on extending this research to the challenging near-field reflector problem. Another promising direction is to apply the $c$-transform technique to enhance the stability of the method, enabling it to handle highly contrasting singular cases \cite{doskolovich2022supporting} more effectively.


\begin{thebibliography}{50}

\bibitem{awanou2015standard}
G. Awanou, Standard finite elements for the numerical resolution of the elliptic Monge–Ampére: classical solutions, IMA Journal of Numerical Analysis 35.3 (2015): 1150-1166.

\bibitem{bao2024optimal}
G. Bao, and Y. Zhang, An optimal transport approach for 3D electrical impedance tomography, Inverse Problems 40.12 (2024): 125006.

\bibitem{benamou2014numerical}
J. D. Benamou, B. D. Froese, A. M. Oberman, Numerical solution of the optimal transportation problem using the Monge–Ampère equation, Journal of Computational Physics 260 (2014): 107-126.


\bibitem{benamou2020entropic}
J. D. Benamou, W. L. Ijzerman and G. Rukhaia, An entropic optimal transport numerical approach to the reflector problem, Methods and Applications of Analysis (2020).


\bibitem{berman2020sinkhorn}
R. J. Berman, The Sinkhorn algorithm, parabolic optimal transport and geometric Monge–Ampére
equations, Numerische Mathematik 145.4 (2020): 771-836.


\bibitem{bosel2017single}
C. Bösel, and H. Gross, Single freeform surface design for prescribed input wavefront and target irradiance, J. Opt. Soc. Am. A 34.9 (2017): 1490-1499.


\bibitem{brenier1991ploar}
Y. Brenier, Polar factorization and monotone rearrangement of vector-valued functions, Comm. Pure Appl. Math. 44 (1991), 375–417.

\bibitem{brix2015solving}
K. Brix, Y. Hafizogullari, and A. Platen, Solving the Monge–Ampère equations for the inverse reflector problem, Mathematical Models and Methods in Applied Sciences 25.05 (2015): 803-837.



\bibitem{bykov2018linear}
D. A. Bykov, L. L. Doskolovich, A. A. Mingazov, et al., Linear assignment problem in the design of freeform refractive optical elements generating prescribed irradiance distributions, Optics Express 26.21 (2018): 27812-27825.


\bibitem{caffarelli1990interior}
L. A. Caffarelli, Interior $W^{2,p}$ estimates for solutions of the Monge-Ampere equation, Annals of Mathematics 131.1 (1990): 135-150.


\bibitem{caffarelli2008regularity}
L. A. Caffarelli, C. E. Gutiérrez, and Q. Huang, On the regularity of reflector antennas, Annals of Mathematics (2008): 299-323.



\bibitem{de2016far}
P. M. M. De Castro, Q. Mérigot, and B. Thibert, Far-field reflector problem and intersection of paraboloids, Numerische Mathematik 134 (2016): 389-411.


\bibitem{desnijder2019ray}
K. Desnijder, P. Hanselaer, and Y. Meuret, Ray mapping method for off-axis and non-paraxial freeform illumination lens design, Optics letters 44.4 (2019): 771-774.


\bibitem{doskolovich2022supporting}
L. L. Doskolovich, E. V. Byzov, A. A. Mingazov, et al., Supporting quadric method for designing freeform mirrors that generate prescribed near-field irradiance distributions, Photonics. Vol. 9. No. 2. MDPI, 2022.



\bibitem{Evans}
L. C. Evans, Partial Differential Equations, Vol. 19. American Mathematical Society, 2022.




\bibitem{feng2016freeform}
Z. Feng, B. D. Froese, and R. Liang, Freeform illumination optics construction following an optimal transport map, Applied Optics 55.16 (2016): 4301-4306.


\bibitem{feng2009mixed}
X. Feng, and M. Neilan, Mixed finite element methods for the fully nonlinear Monge–Ampère equation based on the vanishing moment method, SIAM Journal on Numerical Analysis 47.2 (2009): 1226-1250.

\bibitem{Fournier2010fast}
F. R. Fournier, W. J. Cassarly, and J. P. Rolland, Fast freeform reflector generation using source-target maps,  Optics Express 18.5 (2010): 5295-5304.


\bibitem{glimm2003optical}
T. Glimm, and V. Oliker, Optical design of single reflector systems and the Monge–Kantorovich mass transfer problem, Journal of Mathematical Sciences 117.3 (2003): 4096-4108.



\bibitem{hamfeldt2021convergent}
B. F. Hamfeldt, and A. G. Turnquist, Convergent numerical method for the reflector antenna problem via optimal transport on the sphere, J. Opt. Soc. Am. A 38.11 (2021): 1704-1713.



\bibitem{kim2012parabolic}
Y. H. Kim, J. Streets, and M. Warren, Parabolic optimal transport equations on manifolds, International Mathematics Research Notices 2012.19 (2012): 4325-4350.



\bibitem{kochengin1998determination}
S. A. Kochengin, and V. I. Oliker, Determination of reflector surfaces from near-field scattering data II. Numerical solution, Numerische Mathematik 79.4 (1998): 553-568.


\bibitem{kochengin2003computational}
S. A. Kochengin, and V. I. Oliker, Computational algorithms for constructing reflectors, Computing and Visualization in Science 6.1 (2003): 15-21.



\bibitem{loeper2005numerical}
G. Loeper, and F. Rapetti, Numerical solution of the Monge–Ampère equation by a Newton's algorithm, Comptes Rendus Mathematique 340.4 (2005): 319-324.


\bibitem{loeper2009regularity}
G. Loeper, On the regularity of solutions of optimal transportation problems, Acta Mathematica 202.2 (2009): 241-283.


\bibitem{loeper2011regularity}
G. Loeper, Regularity of optimal maps on the sphere: The quadratic cost and the reflector antenna,  Archive for Rational Mechanics and Analysis 199 (2011): 269-289.



\bibitem{ma2005regularity}
X. N. Ma, N. S. Trudinger, and X. J. Wang, Regularity of potential functions of the optimal transportation problem, Archive for Rational Mechanics and Analysis 177 (2005): 151-183.



\bibitem{mccann2001polar}
R. J. McCann, Polar factorization of maps on Riemannian manifolds, Geometric \& Functional Analysis GAFA 11.3 (2001): 589-608.

\bibitem{mcrae2018optimal}
A. T. McRae, C. J. Cotter, and C. J. Budd, Optimal-transport-based mesh adaptivity on the plane and sphere using finite elements, SIAM Journal on Scientific Computing 40.2 (2018): A1121-A1148.


\bibitem{meyron2018light}
J. Meyron, Q. Mérigot, and B. Thibert, Light in power: a general and parameter-free algorithm for caustic design, ACM Transactions on Graphics (TOG) 37.6 (2018): 1-13.



\bibitem{Neuberger}
J. Neuberger, Sobolev gradients and differential equations, Springer Science \& Business Media, 2009.



\bibitem{prins2015least}
C. R. Prins, R. Beltman, J. H. M. ten Thije Boonkkamp, W. L. IJzerman, and T. W. Tukker, A least-squares method for optimal transport using the Monge-Ampère equation, SIAM Journal on Scientific Computing 37.6 (2015): B937-B961.


\bibitem{fire}
F. Rathgeber, D. A. Ham, L. Mitchell, et al., Firedrake: Automating the finite element method by composing abstractions, ACM Trans. Math. Software, 43 (2017), 24.


\bibitem{romijn2021iterative}
L. B. Romijn, J. H. M. ten Thije Boonkkamp, M. J. H. Anthonissen, and W. L. IJzerman, An iterative least-squares method for generated Jacobian equations in freeform optical design, SIAM Journal on Scientific Computing 43.2 (2021): B298-B322.


\bibitem{sulman2011efficient}
M. M. Sulman, J. F. Williams, and R. D. Russell, An efficient approach for the numerical solution of the Monge–Ampère equation, Applied Numerical Mathematics 61.3 (2011): 298-307.


\bibitem{villani2009optimal}
C. Villani, Optimal transport: old and new. Vol. 338. Berlin: Springer, 2009.

\bibitem{wang1996design}
X. J. Wang, On the design of a reflector antenna, Inverse Problems 12.3 (1996): 351.


\bibitem{wang2004design}
X. J. Wang, On the design of a reflector antenna II, Calculus of Variations and Partial Differential Equations 20.3 (2004): 329-341.


\bibitem{wu2013freeform}
R. Wu, L. Xu, P. Liu, Y. Zhang, Z. Zheng, H. Li, and X. Liu, Freeform illumination design: a nonlinear boundary problem for the elliptic Monge–Ampére equation, Optics Letters 38.2 (2013): 229-231.






\end{thebibliography}
\end{document}